\documentclass[11pt,a4paper]{article}

\usepackage{titlesec}
\usepackage{fancyhdr}
\usepackage{a4wide}
\usepackage{graphicx}
\usepackage{float}
\usepackage{amssymb}
\usepackage{amsmath}
\usepackage{amsthm}
\usepackage{color}
\usepackage{mathrsfs}
\usepackage{array}
\usepackage{eucal}
\usepackage{tikz}
\usepackage{multido}
\usepackage{wrapfig}
\usepackage[T1]{fontenc}
\usepackage{inputenc}
\usepackage[english]{babel}

\usepackage{lmodern}
\usepackage{hyperref}
\usepackage{geometry}
\usepackage{changepage}
\changepage{0pt}{}{}{}{}{0pt}{}{0pt}{10pt}
\usepackage[numbers]{natbib}
\usepackage{hyperref}
\hypersetup{
pdfpagemode=none,
pdftoolbar=true,        
pdfmenubar=true,        
pdffitwindow=false,     
pdfstartview={Fit},    
pdftitle={Invisibility and perfect reflectivity in waveguides with finite length branches},    
pdfauthor={L. Chesnel, S.A. Nazarov, V. Pagneux},     
pdfsubject={},  
pdfcreator={L. Chesnel, S.A. Nazarov, V. Pagneux},   
pdfproducer={L. Chesnel, S.A. Nazarov, V. Pagneux}, 
pdfkeywords={}, 
pdfnewwindow=true,      
colorlinks=true,       
linkcolor=magenta,          
citecolor=red,        
filecolor=cyan,      
urlcolor=blue           
}

\newcommand{\dsp}{\displaystyle}

\newcommand{\eps}{\varepsilon}
\newcommand{\om}{\omega}
\newcommand{\Om}{\Omega}
\newcommand{\mrm}[1]{\mathrm{#1}}

\newcommand{\Cplx}{\mathbb{C}}
\newcommand{\N}{\mathbb{N}}
\newcommand{\R}{\mathbb{R}}

\newcommand{\mL}{\mrm{L}}
\newcommand{\mH}{\mrm{H}}

\newcommand{\mA}{\mathscr{A}}

\newcommand{\mX}{\mrm{X}}

\newcommand{\loc}{\mbox{\scriptsize loc}}

\newcommand{\rcoef}{\mathcal{R}}
\newcommand{\tcoef}{\mathcal{T}} 
\newcommand{\rN}{r}
\newcommand{\rD}{R}

\newtheorem{proposition}{Proposition}[section]

\begin{document}

~\vspace{0.0cm}
\begin{center}
{\sc \bf\LARGE  Invisibility and perfect reflectivity in}\\[6pt]
{\sc \bf\LARGE  waveguides with finite length branches}
\end{center}

\begin{center}
\textsc{Lucas Chesnel}$^1$, \textsc{Sergei A. Nazarov}$^{2,\,3,\,4}$, \textsc{Vincent Pagneux}$^5$\\[16pt]
\begin{minipage}{0.96\textwidth}
{\small
$^1$ INRIA/Centre de mathématiques appliquées, \'Ecole Polytechnique, Université Paris-Saclay, Route de Saclay, 91128 Palaiseau, France;\\
$^2$ St. Petersburg State University, Universitetskaya naberezhnaya, 7-9, 199034, St. Petersburg, Russia;\\
$^3$ Peter the Great St. Petersburg Polytechnic University, Polytekhnicheskaya ul, 29, 195251, St. Petersburg, Russia;\\
$^4$ Institute of Problems of Mechanical Engineering, Bolshoy prospekt, 61, 199178, V.O., St. Petersburg, Russia;\\
$^5$ Laboratoire d'Acoustique de l'Université du Maine, Av. Olivier Messiaen, 72085 Le Mans, France.\\[10pt]
E-mails: \texttt{lucas.chesnel@inria.fr}, \texttt{srgnazarov@yahoo.co.uk}, \texttt{vincent.pagneux@univ-lemans.fr}\\[-14pt]
\begin{center}
(\today)
\end{center}
}
\end{minipage}
\end{center}
\vspace{0.4cm}

\noindent\textbf{Abstract.} 
We consider a time-harmonic wave problem, appearing for example in water-waves theory, in acoustics or in electromagnetism, in a setting such that the analysis reduces to the study of a 2D waveguide problem with a Neumann boundary condition. The geometry is symmetric with respect to an axis orthogonal to the direction of propagation of waves. Moreover, the waveguide contains one branch of finite length. We analyse the behaviour of the complex scattering coefficients $\rcoef$, $\tcoef$ as the length of the branch increases and we show how to design geometries where non reflectivity ($\rcoef=0$, $|\tcoef|=1$), perfect reflectivity ($|\rcoef|=1$, $\tcoef=0$) or perfect invisibility ($\rcoef=0$, $\tcoef=1$) hold. Numerical experiments illustrate the different results.\\


\section{Introduction}\label{Introduction}

Invisibility is an exciting topic in scattering theory. In the present article, we consider a time-harmonic waves problem in a 2D waveguide unbounded in one direction (say $(Ox)$) with a non-penetration (Neumann) boundary condition. This problem appears naturally in acoustics, in water-waves theory (for straight vertical walls and horizontal bottom) or in electromagnetism. In the waveguide geometry, at a given frequency, only a finite number of waves can propagate along the $(Ox)$ axis. More precisely, the quantity of interest (the acoustic pressure, the velocity potential, ...) decomposes at $\pm\infty$ as the sum of a finite number of propagating waves plus an infinite number of exponentially decaying modes. All through the paper, we will assume that the frequency is small enough so that only one wave (the piston wave) can propagate in the 2D waveguide. To describe the scattering process of the incident piston wave coming from $-\infty$, classically one introduces two complex coefficients, namely the \textit{reflection} and \textit{transmission} coefficients, denoted $\rcoef$ and $\tcoef$, such that $\rcoef$ (resp. $\tcoef$) corresponds to the amplitude of the scattered field at $-\infty$ (resp. $+\infty$). According to the energy conservation, we have
\begin{equation}\label{EnergyConservation}
|\rcoef|^2+|\tcoef|^2=1.
\end{equation} 
In this work, we are interested in geometries where non reflectivity ($\rcoef=0$), perfect reflectivity ($\tcoef=0$) or perfect invisibility ($\tcoef=1$) occurs. Of course, due to the conservation of energy (\ref{EnergyConservation}), perfect invisibility implies non reflectivity. The converse is wrong since we can have $|\tcoef|=1$ with $\tcoef\ne1$. In this case, the incident piston wave goes through the waveguide with a phase shift.\\
\newline 
In this setting, examples of situations where quasi invisibility ($|\rcoef|$ small or $|\tcoef-1|$ small) happens, obtained via numerical simulations, exist in literature. We refer the reader to \cite{PoNe14,EvMP14} for water wave problems and to \cite{AlSE08,EASE09,NgCH10,OuMP13,FuXC14} for strategies based on the use of new ``zero-index'' and ``epsilon near zero'' metamaterials in electromagnetism (see \cite{FlAl13} for an application to acoustic). Let us mention also that the problem of the existence of quasi invisible obstacles for frequencies close to the threshold frequency has been addressed in the analysis of the so-called Weinstein anomalies \cite{Vain66} (see e.g. \cite{na580,KoNS16}).\\
\newline
As for the rigorous proof of existence of geometries where $\rcoef=0$ or $\tcoef=1$, literature is not very developed especially if we compare to what is available concerning the existence of \textit{trapped modes} (see e.g. \cite{Urse51,Evan92,EvLV94,DaPa98,LiMc07,Naza10c,NaVi10}). We remind the reader that trapped modes are non zero solutions to the homogeneous problem (\ref{PbInitial}) which are exponentially decaying both at $\pm\infty$. Using a similar terminology, we can call \textit{invisible modes} the solutions of (\ref{PbInitial}) such that the scattered field is exponentially decaying both at $\pm\infty$ ($\tcoef=1$). Such a difference of treatment between trapped and invisible modes in literature is striking since the two notions seem to share similarities.\\
\newline
One way to find situations where $\rcoef=0$ or $\tcoef=0$ (but not $\tcoef=1$) is to use the so-called Fano resonance (see the seminal paper \cite{Fano61}). Let us present briefly the idea which is developed and justified under some assumptions that can be verified numerically in \cite{ShVe05,ShTu12,ShWe13,AbSh16} in the context of gratings in electromagnetism. If for a given wavenumber $k_0$ there is a trapped mode, then perturbing slightly the geometry allows one to exhibit settings where the scattering coefficients have a fast variation for $k$ moving on the real axis around $k_0$. And with additional geometric assumptions, one can show that $|\rcoef|$, $|\tcoef|$ pass through $0$ and $1$. For waveguides problems, we refer the reader to \cite{ChNaSu}.\\
\newline
Another approach to construct waveguides such that $\rcoef=0$ has been proposed in \cite{BoNa13,BLMN15} (see also \cite{BoNTSu,BoCNSu,ChHS15,ChNa16} for applications to other problems).  The method  consists in adapting the proof of the implicit functions theorem. More precisely, the idea is to observe that $\rcoef=0$ in the straight waveguide and then to make a well-chosen smooth perturbation of amplitude $\eps$ (small) in the boundary to keep $\rcoef=0$. As explained in \cite{BoNa13}, this strategy does not permit to impose $\tcoef=1$ (perfect invisibility) for waveguides with Neumann boundary conditions because the differential of $\tcoef$ with respect to the deformation for the reference geometry is not onto in $\Cplx$ (think to the assumptions of the implicit functions theorem). However, this problem was overcome in \cite{BoCN18} where it is shown how to get $\tcoef=1$ (and not only $|\tcoef|=1$) working with singular perturbations (instead of smooth ones) made of thin rectangles. Let us mention that these types of techniques proposed in \cite{Naza11,Naza13} were used in \cite{Naza11c,Naza12,CaNR12,Naza11b} in a similar context. In these works, the authors construct small (non necessarily symmetric) perturbations of the walls of a waveguide that preserve the presence of a trapped mode associated with an eigenvalue embedded in the continuous spectrum.\\
\newline
It is important to emphasize that the methods of the previous paragraph are perturbative methods. They require to start from a geometry where it is known that $\rcoef=0$ or/and $\tcoef=1$. In our case, this geometry is simply the reference (straight) waveguide. As a consequence, the technique cannot be used to construct waveguides where $\tcoef=0$ (perfect reflectivity). In this article, we propose to investigate another route allowing us to get $\rcoef=0$, $\tcoef=0$ and also $\tcoef=1$. It relies on two main ingredients: symmetries and asymptotic analysis for truncated waveguides. Interestingly, our approach provides examples of geometries where $\rcoef=0$ or $\tcoef=1$ which are not small perturbations of the reference waveguide. In our study, we will be led to consider scattering problems in $\sf{T}$-shaped waveguides. Such problems have been considered in particular in \cite{Naza10a,NaSh11}. Let us mention also that this work shares connections with \cite{CaLo08,Mois09,BuSa11,HeKN12}. In the latter papers, the authors investigate the presence of trapped modes (also called bound states) associated with eigenvalues embedded in the continuous spectrum in geometries similar to ours. Finally, note that in the present article, we deal only with the Neumann boundary conditions. However, Dirichlet waveguides can be treated similarly and analogous results would be obtained. We emphasize that our approach is exact in the sense that we do not neglect exponentially decaying modes (this simplifying assumption appears very often in physics literature). \\
\newline
The paper is structured as follows. We begin by introducing the setting and notation in Section \ref{SectionSetting}. The waveguide $\Om_L$ is symmetric with respect to the $(Oy)$ axis (perpendicular to the unbounded direction) and contains one vertical (along the $(Oy)$ axis) branch of finite length $L-1$. Using the symmetry, we decompose the problem into two sub-problems set in half-waveguides with different boundary conditions: one with Neumann boundary conditions, another with mixed (Dirichlet and Neumann) boundary conditions. Then, we compute an asymptotic expansion of the scattering coefficients $\rcoef$, $\tcoef$ as $L\to+\infty$ (the branch of finite length becomes longer and longer). This expansion depends on the number of propagating modes existing in the vertical branch of the unbounded $\sf{T}$-shaped waveguide $\Om_{\infty}$ obtained at the limit $L=+\infty$, and this number itself depends on the width $\ell$ of the vertical branch of $\Om_{\infty}$. In Section \ref{SectionAsymptotic}, we focus our attention on small values of $\ell$ for which only one propagating mode exists in the vertical branch of $\Om_{\infty}$. In Section \ref{SectionNonReflection}, we use the asymptotic expansions of the scattering coefficients to prove the existence of geometries where one has $\rcoef=0$ (non reflectivity) or $\tcoef=0$ (perfect reflectivity). In Section \ref{SectionTwoModes}, we consider a larger value for the parameter $\ell$ such that two modes can propagate in the vertical branch of $\Om_{\infty}$. In such cases, we show that the behaviour of the scattering coefficients as $L\to+\infty$ can be quite complex. In Section \ref{SectionCompleteInvisibility}, we explain how to construct waveguides where there holds $\tcoef=1$ (perfect invisibility). In Section \ref{SectionNumerics} we provide numerical experiments illustrating the different results obtained in the paper. Finally, in Section \ref{sectionConclusion}, we give a brief conclusion and in the Appendix we gather the proof of two statements used in the analysis. The main results of this article are Proposition \ref{ThinBranch} (non reflectivity and perfect reflectivity for a thin vertical branch), Proposition \ref{WideBranch} (non reflectivity and perfect reflectivity for a larger vertical branch) and the approach presented in Section \ref{SectionCompleteInvisibility} to obtain perfect invisibility.

\section{Setting}\label{SectionSetting}

\begin{figure}[!ht]
\centering
\begin{tikzpicture}[scale=1.5]
\draw[fill=gray!30,draw=none](-2,1) rectangle (2,2);
\draw[fill=gray!30,draw=none](-5,1) rectangle (-2,2);
\draw[fill=gray!30,draw=none](-2,1.8) rectangle (-1,3.8);
\draw (-2,1)--(2,1); 
\draw (-5,2)--(-2,2)--(-2,3.8)--(-1,3.8)--(-1,2)--(2,2); 
\draw (-5,1)--(-2,1); 
\draw[dashed] (3,1)--(2,1); 
\draw[dashed] (3,2)--(2,2);
\draw[dashed] (-6,2)--(-5,2);
\draw[dashed] (-6,1)--(-5,1);
\draw[->] (3,1.2)--(3.6,1.2);
\draw[->] (3.1,1.1)--(3.1,1.7);
\draw[dotted,>-<] (-2,4)--(-1,4);
\draw[dotted,>-<] (-0.8,1.95)--(-0.8,3.85);
\node at (3.65,1.3){\small $x$};
\node at (3.25,1.6){\small $y$};
\node at (-1.5,4.2){\small $\ell$};
\node at (-0.4,2.9){\small $L-1$};
\draw[dashed,line width=0.5mm,gray!80] (-1.5,1)--(-1.5,3.8);
\draw[fill=white] (-0.6,1.5) ellipse (0.4 and 0.3);
\draw[fill=white] (-2.4,1.5) ellipse (0.4 and 0.3);
\begin{scope}[xshift=-4.6cm,yshift=1.7cm,scale=0.8]
\draw[line width=0.2mm,->] plot[domain=0:pi/4,samples=100] (\x,{0.2*sin(20*\x r)}) node[anchor=west] {\hspace{-2.4cm}$1$};
\end{scope}
\begin{scope}[xshift=-4.6cm,yshift=1.3cm,scale=0.8]
\draw[line width=0.2mm,<-] plot[domain=0:pi/4,samples=100] (\x,{0.2*sin(20*\x r)}) node[anchor=west] {$\hspace{-2.4cm}\mathcal{R}$};
\end{scope}
\begin{scope}[xshift=1cm,yshift=1.5cm,scale=0.8]
\draw[line width=0.2mm,->] plot[domain=0:pi/4,samples=100] (\x,{0.2*sin(20*\x r)}) node[anchor=west] {$\mathcal{T}$};
\end{scope}

\end{tikzpicture}
\caption{Example of geometry $\Om_L$. The vertical thick dashed line marks the axis of symmetry of the domain which will play a key role in the analysis. \label{DomainOriginal2D}} 
\end{figure}
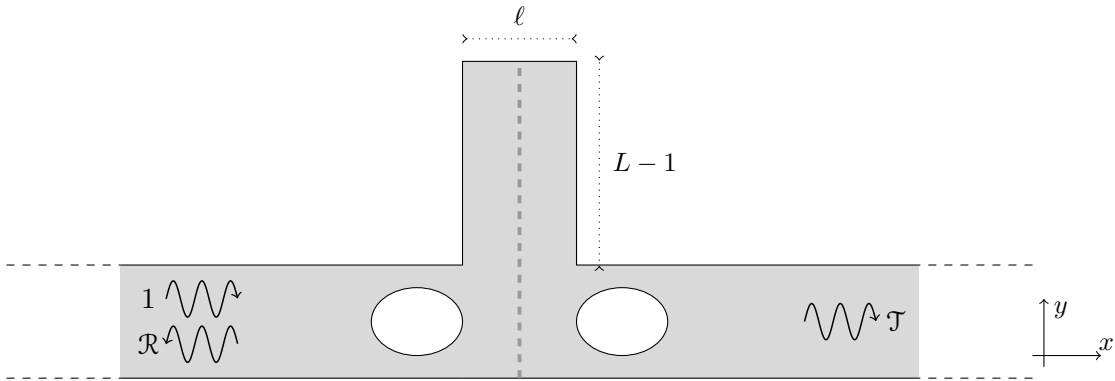

For $\ell>0$, $L>1$, consider a connected open set $\Om_L\subset\R^2$ (see Figure \ref{DomainOriginal2D}) which coincides with the region 
\[
\{ (x,y)\in\R\times(0;1)\ \cup\  (-\ell/2;\ell/2)\times [1;L)\}
\]
outside a given ball centered at $0$ of radius $d>0$ (independent of $\ell$, $L$). We assume that $\Om_L$ is symmetric with respect to the $(Oy)$ axis ($\Om_L=\{(-x,y)\,|\,(x,y)\in\Om_L\}$) and that its boundary $\partial\Om_L$ is Lipschitz. We work in a rather academic geometry but other settings can be considered as well (see Figures \ref{figResultAngleDroit1}, \ref{figResultAngleDroit2} and the discussion in Section \ref{sectionConclusion}). We assume that the propagation of time-harmonic waves in $\Om_L$ is governed by the Helmholtz equation with Neumann boundary conditions
\begin{equation}\label{PbInitial}
\begin{array}{|rcll}
\Delta v + k^2 v & = & 0 & \mbox{ in }\Om_L\\[3pt]
 \partial_nv  & = & 0  & \mbox{ on }\partial\Om_L. 
\end{array}
\end{equation}
In this problem, $\Delta$ denotes the 2D Laplace operator, $k$ is the wavenumber and $n$ stands for the normal unit vector to $\partial\Om_L$ directed to the exterior of $\Om_L$. Moreover, $v$ corresponds for example to the velocity potential in water-waves theory or to the pressure in acoustics. We assume that $k\in(0;\pi)$ so that $k^2$ is located between the first and second thresholds of the continuous spectrum $\sigma_c=[0;+\infty)$ of Problem (\ref{PbInitial}) and we set 
\[
w^{\pm}(x,y)=e^{\pm ik x}/\sqrt{2k}.
\]
In the following, $w^{\pm}$ will serve to define the incident and scattered fields. Introduce $\chi^{+}\in\mathscr{C}^{\infty}(\R^2)$ (resp. $\chi^{-}\in\mathscr{C}^{\infty}(\R^2)$) a cut-off function that is equal to one for $x\ge 2\max(\ell,d)$ (resp. $x\le -2\max(\ell,d)$) and to zero for $x\le \max(\ell,d)$ (resp. $x\ge -\max(\ell,d)$). The scattering process of the incident piston wave coming from $-\infty$ by the structure is described by the problem 
\begin{equation}\label{PbChampTotal}
\begin{array}{|rcll}
\multicolumn{4}{|l}{\mbox{Find }v\in\mrm{H}^1_{\loc}(\Om_L) \mbox{ such that }v-\chi^-w^+\mbox{ is outgoing and } }\\[3pt]
\Delta v +k^2 v & = & 0 & \mbox{ in }\Om_L\\[3pt]
 \partial_nv  & = & 0  & \mbox{ on }\partial\Om_L. 
\end{array}
\end{equation}
Here, $v-\chi^-w^+$ is outgoing means that there holds the decomposition (see the schematic picture \ref{DomainOriginal2D})
\begin{equation}\label{defZetaL}
v-\chi^-w^+ = \chi^{-}\rcoef\,w^-+\chi^{+}\tcoef\,w^++\tilde{v}
\end{equation}
with $\tilde{v}\in\mrm{H}^1(\Om_L)$ which is exponentially decaying at $\pm\infty$. One can prove that Problem (\ref{PbChampTotal}) always admits a solution (see e.g. \cite[Chap. 5, \S3.3, Thm. 3.5 p.160]{NaPl94}) which is possibly non uniquely defined if there is a trapped mode\footnote{We remind the reader that we call ``trapped mode'' a solution to Problem (\ref{PbInitial}) which belongs to $\mH^1(\Om_L)$ (see \cite{LiMc07} for more details).} at the wavenumber $k$. However, the reflection coefficient $\rcoef\in\Cplx$ and transmission coefficient $\tcoef\in\Cplx$ are always uniquely defined. They satisfy the energy conservation relation
\[
|\rcoef|^2+|\tcoef|^2=1
\]
already written in (\ref{EnergyConservation}). Of course $\rcoef$ and $\tcoef$ depend on the features of the geometry, in particular on $L$. In this work, we explain how to find some $L$ such that $\rcoef=0$, $|\tcoef|=1$ (non reflectivity); $|\rcoef|=1$, $\tcoef=0$ (perfect reflectivity); or $\rcoef=0$, $\tcoef=1$ (perfect invisibility). To obtain such particular values for the scattering coefficients, we will use the fact that the geometry is symmetric with respect to the $(Oy)$ axis. Define the half-waveguide 
\[
\om_L:=\{(x,y)\in\Om_L\,|\,x<0\}.
\]
(see Figure \ref{LimitDomain}, left). Introduce the problem with Neumann boundary conditions 
\begin{equation}\label{PbChampTotalSym}
\begin{array}{|rcll}
\Delta u +k^2 u & = & 0 & \mbox{ in }\om_L\\[3pt]
 \partial_nu  & = & 0  & \mbox{ on }\partial\om_L
\end{array}
\end{equation}
as well as the problem with mixed boundary conditions 
\begin{equation}\label{PbChampTotalAntiSym}
\begin{array}{|rcll}
\Delta U + k^2 U & = & 0 & \mbox{ in }\om_L\\[3pt]
 \partial_nU  & = & 0  & \mbox{ on }\partial\om_L\cap\partial\Om_L \\[3pt]
U  & = & 0  & \mbox{ on }\Sigma_L:=\partial\om_L\setminus\partial\Om_L.
\end{array}
\end{equation}
Problems (\ref{PbChampTotalSym}) and (\ref{PbChampTotalAntiSym}) admit respectively the solutions 
\begin{equation}\label{defZetaLsym}
u = \chi^-(w^++\rN\,w^-) + \tilde{u},\qquad\mbox{ with }\tilde{u}\in\mH^1(\om_L),
\end{equation} 
\begin{equation}\label{defZetaLanti}
U = \chi^-(w^++\rD\,w^-) + \tilde{U},\qquad\mbox{ with }\tilde{U}\in\mH^1(\om_L),
\end{equation} 
where $\rN$, $\rD\in\Cplx$ are uniquely defined. Moreover, due to conservation of energy, one has
\begin{equation}\label{NRJHalfguide}
|\rN|=|\rD|=1.
\end{equation}
Briefly, let us explain how to show the latter identities. First, integrating by parts, one obtains
\begin{equation}\label{integral}
\int_{x=-\xi} (\partial_nu\,\overline{u}- u\,\partial_n\overline{u})\,dy=0
\end{equation}
for $\xi>0$ large enough. Here, we denote $\partial_n=-\partial_x$ at $x=-\xi$. Observing that the integral (\ref{integral}) does not depend on $\xi$, taking the limit $\xi\to+\infty$ and using the explicit representation (\ref{defZetaLsym}), we get $|\rN|=1$. Working analogously with $U$ and exploiting (\ref{defZetaLanti}) leads to $|\rD|=1$.\\
\newline
Now, direct inspection shows that if $v$ is a solution of Problem (\ref{PbChampTotal}) then, we have $v(x,y)=(u(x,y)+U(x,y))/2$ in $\om_L$ and $v(x,y)=(u(-x,y)-U(-x,y))/2$ in $\Om_L\setminus\overline{\om_L}$ (up possibly to a term which is exponentially decaying at $\pm\infty$ if there is a trapped mode at the given wavenumber $k$). We deduce that the scattering coefficients $\rcoef$, $\tcoef$ appearing in the decomposition (\ref{defZetaL}) of $v$ are such that
\begin{equation}\label{Formulas}
\rcoef=\frac{\rN+\rD}{2}\qquad\mbox{ and }\qquad \tcoef=\frac{\rN-\rD}{2}.
\end{equation}
Imagine that we want to have $\rcoef=0$ (non reflectivity). According to (\ref{Formulas}), we must impose $\rN=-\rD$. Relations (\ref{NRJHalfguide}) guarantee that for all $L>1$, both $\rN$ and $\rD$ are located on the unit circle $\mathbb{S}:=\{z\in\Cplx\,|\,|z|=1\}$. In the following, we will show that for $\ell$, the width of the vertical branch of $\Om_L$, smaller than $\pi/k$, $\rD$ tends to a constant $R_{\infty}\in\mathbb{S}$ while $\rN$ runs continuously along $\mathbb{S}$ as $L\to+\infty$. This will prove the existence of $L$ such that $\rN=-\rD$ and so $\rcoef=0$. This will also show that there is some $L$ such that $\rN=\rD$ and, therefore, $\tcoef=0$ (perfect reflectivity). In order to obtain perfect invisibility, that is $\tcoef=1$, we must impose both $\rN=1$ and $\rD=-1$. In other words, there is an additional constrain to satisfy and we will need to play with another degree of freedom. Here, we do not explain how to proceed, this will be the concern of Section \ref{SectionCompleteInvisibility}. The important outcome of this discussion is that we will study the behaviour of $\rN$ and $\rD$ with respect to $L$ going to $+\infty$. As one can imagine, this behaviour depends on the properties of the equivalents of Problems (\ref{PbChampTotalSym}), (\ref{PbChampTotalAntiSym}) set in the limit geometry $\om_{\infty}$ obtained from $\om_L$ making $L\to+\infty$ 
(see Figure \ref{LimitDomain}, right). More precisely, the number of propagating waves existing in the vertical branch of $\om_{\infty}$ will play a key role in the analysis. 

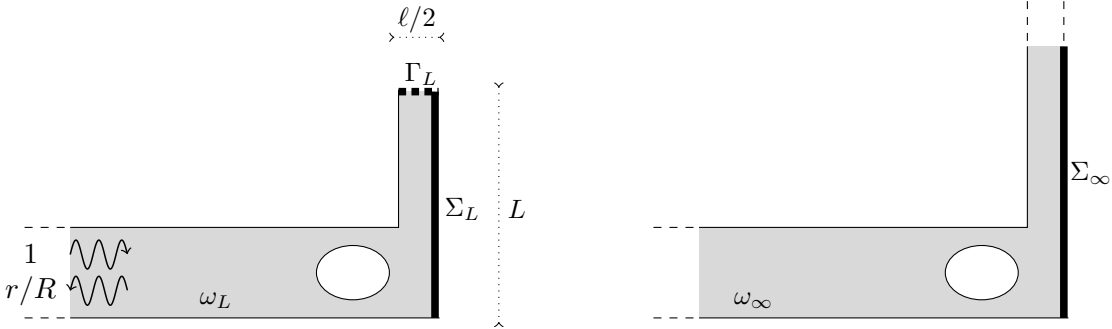
\begin{figure}[!ht]
\centering
\begin{tikzpicture}[scale=1.2]
\draw[fill=gray!30,draw=none](-4,0) rectangle (0,1);
\draw[fill=gray!30,draw=none](-0.4,0) rectangle (0,2.5);
\draw (-4,1)--(-0.4,1)--(-0.4,2.52);
\draw[line width=1mm, dotted] (-0.4,2.5)--(0.04,2.5);
\draw (0.05,0)--(-4,0);
\draw[line width=1mm] (0,0)--(0,2.5);
\draw[dashed] (-4.5,1)--(-4,1); 
\draw[dashed] (-4.5,0)--(-4,0); 
\node at (0.3,1.2){\small $\Sigma_{L}$};
\node at (-2.4,0.2){\small $\om_{L}$};
\draw[dotted,>-<] (-0.5,3.1)--(0.1,3.1);
\draw[dotted,>-<] (0.7,-0.1)--(0.7,2.6);
\node at (0.9,1.2){\small $L$};
\node at (-0.15,2.7){\small $\Gamma_{L}$};
\node at (-0.2,3.3){\small $\ell/2$};
\phantom{\draw[dashed] (0,3.5)--(0,3); }
\draw[fill=white] (-0.9,0.5) ellipse (0.4 and 0.3);
\begin{scope}[xshift=-4cm,yshift=0.7cm,scale=0.8]
\draw[line width=0.2mm,->] plot[domain=0:pi/4,samples=100] (\x,{0.2*sin(20*\x r)}) node[anchor=west] {\hspace{-2.6cm}$1$};
\end{scope}
\begin{scope}[xshift=-4cm,yshift=0.3cm,scale=0.8]
\draw[line width=0.2mm,<-] plot[domain=0:pi/4,samples=100] (\x,{0.2*sin(20*\x r)}) node[anchor=west] {$\hspace{-2.6cm}r/R$};
\end{scope}
\end{tikzpicture}\qquad\qquad\begin{tikzpicture}[scale=1.2]
\draw[fill=gray!30,draw=none](-4,0) rectangle (0,1);
\draw[fill=gray!30,draw=none](-0.4,0) rectangle (0,3);
\draw (-4,0)--(0.05,0); 
\draw[line width=1mm] (0,0)--(0,3);
\draw (-4,1)--(-0.4,1)--(-0.4,3);
\draw[dashed] (-4.5,1)--(-4,1); 
\draw[dashed] (-4.5,0)--(-4,0); 
\draw[dashed] (0,3.5)--(0,3); 
\draw[dashed] (-0.4,3.5)--(-0.4,3); 
\node at (0.3,1.6){\small $\Sigma_{\infty}$};
\node at (-3.4,0.2){\small $\om_{\infty}$};
\phantom{\draw[dotted,>-<] (0.3,-0.1)--(0.3,2.6);}
\draw[fill=white] (-0.9,0.5) ellipse (0.4 and 0.3);
\end{tikzpicture}
\caption{Domains $\om_{L}$ (left) and $\om_{\infty}$ (right).\label{LimitDomain}} 
\end{figure}

\section{Asymptotic expansion of the scattering coefficients as $L\to+\infty$}\label{SectionAsymptotic}

\subsection{Half-waveguide problem with mixed boundary conditions}\label{paragraphMixed}
Consider the problem obtained from (\ref{PbChampTotalAntiSym}) making formally $L\to+\infty$:
\begin{equation}\label{PbUnbounded}
\begin{array}{|rcll}
\Delta U_{\infty} + k^2 U_{\infty}  & = & 0 & \mbox{ in }\om_{\infty}\\[3pt]
 \partial_n U_{\infty}   & = & 0  & \mbox{ on }\partial\om_{\infty}\cap\partial\Om_{\infty}\\[3pt]
U_{\infty}  & = & 0  & \mbox{ on }\Sigma_{\infty}:=\partial\om_{\infty}\setminus\partial\Om_{\infty}.
\end{array}
\end{equation}
Here $\Om_{\infty}$ is the domain obtained from $\Om_L$ with $L\to+\infty$. When $\ell\in(0;\pi/k)$, propagating modes in the vertical branch of $\om_{\infty}$ for Problem (\ref{PbUnbounded}) do not exist and we can show that (\ref{PbUnbounded}) admits the solution 
\begin{equation}\label{ScatteringDemiMixed}
U_{\infty} = \chi^-(w^++R_{\infty}\,w^-) + \tilde{U}_\infty,\qquad\mbox{ with }\tilde{U}_\infty\in\mH^1(\om_\infty),
\end{equation}
where $R_{\infty}$, such that $|R_{\infty}|=1$, (work as in (\ref{NRJHalfguide}) to establish this identity), is uniquely defined. In Proposition \ref{PropoAsymptotic} in Appendix, when Problem (\ref{PbUnbounded}) admits only the zero solution in $\mH^1(\om_{\infty})$ (absence of trapped modes), we explain how to prove the expansion 
\[
U = U_{\infty}+ \dots\,
\]
(see the precise statement in (\ref{VolumicEstimate})). Here and in what follows, the dots correspond to a remainder which is exponentially small as $L\to+\infty$. Hence, we deduce 
\begin{equation}\label{Resultat0mode}
\rD = R_{\infty}+ \dots\,.
\end{equation}
More precisely, we can show that $|\rD - R_{\infty} | \le C\,e^{-\sqrt{(\pi/\ell)^2-k^2} L}$ where $C$ is independent of $L$ (estimate (\ref{MainEstimationCoef}) in Proposition \ref{PropoAsymptotic} in Appendix).

\subsection{Half-waveguide problem with Neumann boundary conditions}\label{paragraphNeumann}

Making $L\to+\infty$ in (\ref{PbChampTotalSym}) leads to the problem 
\begin{equation}\label{PbUnboundedBis}
\begin{array}{|rcll}
\Delta u_{\infty} + k^2 u_{\infty} & = & 0 & \mbox{ in }\om_{\infty}\\[3pt]
 \partial_nu_{\infty}  & = & 0  & \mbox{ on }\partial\om_{\infty}.
\end{array}
\end{equation} 
When $\ell\in(0;2\pi/k)$, one propagating mode exists in the vertical branch of $\om_{\infty}$ for (\ref{PbUnboundedBis}). Set 
\[
w_{\circ}^{\pm}(x,y)=e^{\pm i k y}/\sqrt{k\ell}.
\]
Problem (\ref{PbUnboundedBis}) admits the solutions
\begin{equation}\label{decompoUnbounded}
\begin{array}{lcl}
u^-_{\infty} & = & \chi^-(w^++r_{\infty}\,w^-)+ \chi^{\circ}\,t_{\infty}\,w^+_{\circ} + \tilde{u}^-_\infty,\\[3pt]
u^\circ_{\infty} & = & \chi^-\,t^{\circ}_{\infty}\,w^- +\chi^{\circ}(w^-_\circ+r^{\circ}_{\infty}\,w^+_\circ) + \tilde{u}^\circ_\infty,
\end{array}
\end{equation}
where $\tilde{u}^-_\infty$, $\tilde{u}^{\circ}_\infty$ are functions in $\mH^1(\om_\infty)$ and where $\chi^{\circ}\in\mathscr{C}^{\infty}(\R^2)$ is such that $\chi^{\circ}=0$ for $y\le 1$, $\chi^{\circ}=1$ for $y\ge 1+\delta$ ($\delta>0$ is a constant). Note that $u^\circ_{\infty}$ is the total field corresponding to an incident wave of unit amplitude which travels in the negative $y-$direction. We define the scattering matrix
\begin{equation}\label{UnboundedScatteringMatrix}
\mathfrak{s}_{\infty}:=\left(\begin{array}{cc}
r_{\infty} & t_{\infty} \\
t_{\infty}^{\circ} & r_{\infty}^{\circ} 
\end{array}\right).
\end{equation}
It is known that $\mathfrak{s}_{\infty}$ is unitary ($\mathfrak{s}_{\infty}\,\overline{\mathfrak{s}_{\infty}}^{\top}=\mrm{Id_{2\times2}}$) and symmetric ($\mathfrak{s}_{\infty}=\mathfrak{s}_{\infty}^{\top}$). For the convenience of the reader, we recall the proof of this result in the Appendix (see Proposition \ref{ProofUnitary}). To obtain an asymptotic expansion of $\rN$ as $L$ goes to $+\infty$, let us compute an asymptotic expansion of $u$. For $u$, we make the ansatz \cite[Chap. 5, \S5.6]{MaNP00}
\begin{equation}\label{Ansatz}
u = u_{\infty}^- + a(L)\,u_{\infty}^{\circ}+\dots
\end{equation}
where $a(L)$ is a gauge function, depending on $L$ but not on $(x,y)$, which has to be determined. On the segment $\Gamma_L:=(-\ell/2;0)\times\{L\}$, we find 
\[
\begin{array}{lcl}
\partial_n u(x,L) & = & \partial_nu_{\infty}^-(x,L) + a(L)\,\partial_nu_{\infty}^{\circ}(x,L)+\dots \\[3pt]
&=& ik\,(k\ell)^{-1/2}\,(\,t_{\infty}\,e^{ik L}+a(L)\,(-e^{- ik L}+r^{\circ}_{\infty}\,e^{ik L})\,)+\dots \ .
\end{array}
\]
Since $\partial_nu=0$ on $\Gamma_L$, we take 
\begin{equation}\label{gaugeFunction}
a(L)=\frac{-t_{\infty}}{-e^{- 2ik L}+r^{\circ}_{\infty}}.
\end{equation}
In order $a(L)$ to be defined for all $L>1$, we must have $|r^{\circ}_{\infty}|\ne1$. Since $\mathfrak{s}_{\infty}$ is unitary and symmetric, this is equivalent to have $t_{\infty}=t^{\circ}_{\infty}\ne0$. If $t_{\infty}=0$, we can choose $a(L)=0$ and prove that $\rcoef = r_{\infty}+\dots$. When $t_{\infty}\ne0$ (so that there is some transmission of energy between the two leads of $\om_{\infty}$ for Problem (\ref{PbUnboundedBis})), plugging expression (\ref{gaugeFunction}) in (\ref{Ansatz}) and identifying the main contribution of the terms of each side of the equality as $x\to-\infty$, we get 
\begin{equation}\label{formulaAsymptotics}
\rN = r_{\mrm{asy}}(L)+\dots\qquad\mbox{ with }\ r_{\mrm{asy}}(L):=r_{\infty}-\dsp\frac{(t^{\circ}_{\infty})^2}{-e^{-2ikL}+r^{\circ}_{\infty}}.
\end{equation}
In (\ref{formulaAsymptotics}), the subscript ``$_{\mrm{asy}}$'' stands for ``asymptotic'' (and not ``asymmetric''). Note that the rigorous demonstration of (\ref{formulaAsymptotics}), which requires to assume that Problem (\ref{PbUnboundedBis}) admits only the zero solution in $\mH^1(\om_{\infty})$ (absence of trapped modes), follows the lines of the proof of Proposition \ref{PropoAsymptotic} in Appendix. However a bit more work is needed to establish a stability estimate corresponding to (\ref{StabilityEstimate}) in this case. To proceed, it is necessary to work  with techniques of weighted Sobolev spaces with detached asymptotics. For
more details, we refer the reader to \cite{Naza13}. As $L$ tends to $+\infty$, the term $r_{\mrm{asy}}(L)$ runs along the set
\begin{equation}\label{setSPlusMoins}
\{r_{\infty}-\dsp\frac{(t^{\circ}_{\infty})^2}{z+r^{\circ}_{\infty}} \,|\,z\in\mathbb{S}\}\qquad\mbox{with}\quad\mathbb{S}=\{z\in\Cplx\,|\,|z|=1\}.
\end{equation}
Using classical results concerning the M\"{o}bius transform, one finds that this set coincides with the circle centered at 
\begin{equation}\label{eqnCenters}
r_{\infty}+\dsp\frac{(t^{\circ}_{\infty})^2\,\overline{r^{\circ}_{\infty}}}{1-|r^{\circ}_{\infty}|^{2}}
\end{equation}
of radius 
\begin{equation}\label{eqnRadius}
\dsp\frac{|t^{\circ}_{\infty}|^2}{1-|r^{\circ}_{\infty}|^{2}}.
\end{equation}
Since $\mathfrak{s}_{\infty}$ is unitary, we have $1=|t^{\circ}_{\infty}|^{2}+|r^{\circ}_{\infty}|^{2}$ and $r^{\circ}_{\infty}=-t^{\circ}_{\infty}\overline{r_{\infty}}/\overline{t^{\circ}_{\infty}}$. From these relations, one can prove that the set defined in (\ref{setSPlusMoins}) is nothing else but the unit circle $\mathbb{S}$.\\
\newline
Since the dots in (\ref{formulaAsymptotics}) correspond to terms which are exponentially decaying as $L$ tends to $+\infty$, we infer that the coefficient $\rN $ does not converge when $L\to+\infty$. Instead, asymptotically as $L\to+\infty$, it behaves like $r_{\mrm{asy}}(L)$, \textit{i.e.} it runs almost periodically along the unit circle $\mathbb{S}$. Since $|\rN|=1$ for all $L\ge 1$, we deduce that $\rN$ also runs (almost periodically) along $\mathbb{S}$ as $L\to+\infty$. The period, which is equal to $\pi/k$, tends to $+\infty$ when $k\to0$. 

\subsection{Original problem}\label{paragraphInitialPb}
From Formula (\ref{Formulas}), we know that the coefficients $\rcoef$, $\tcoef$ appearing in the decomposition (\ref{defZetaL}) of a solution to Problem (\ref{PbChampTotal}) set in $\Om_L$ satisfy $\rcoef=(\rN+\rD)/2$ and $\tcoef=(\rN-\rD)/2$. From the results of \S\ref{paragraphMixed} and \S\ref{paragraphNeumann}, we deduce that when $\ell\in(0;\pi/k)$, we have 
\begin{equation}\label{CircleFirstCase}
\begin{array}{lcl}
\rcoef = \rcoef_{\mrm{asy}}(L)+\dots\qquad\quad &\mbox{ with }\qquad \rcoef_{\mrm{asy}}(L)=(r_{\mrm{asy}}(L) + R_{\infty})/2,\\[3pt]
\tcoef = \tcoef_{\mrm{asy}}(L)+\dots\qquad\quad &\mbox{ with }\qquad \tcoef_{\mrm{asy}}(L)=(r_{\mrm{asy}}(L) - R_{\infty})/2.
\end{array}
\end{equation}
Here $r_{\mrm{asy}}(L)$ is defined in (\ref{formulaAsymptotics}). This shows that asymptotically, $\rcoef$ (resp. $\tcoef$) runs along a circle of radius $1/2$ centered at $R_{\infty}/2$ (resp. $-R_{\infty}/2$).

\section{Non reflectivity and perfect reflectivity for thin vertical branches}\label{SectionNonReflection}
Now we explain how to use the results of the previous section to show the existence of geometries where there holds $\rcoef=0$ (non reflectivity) or $\tcoef=0$ (perfect reflectivity) for the given frequency $k\in(0;\pi)$. We work in the geometry $\Om_L$ introduced in Section \ref{SectionSetting} with $\ell\in(0;\pi/k)$ ($\ell$ is the width of the vertical branch, see Figure \ref{DomainOriginal2D}).
\begin{proposition}\label{ThinBranch}
Assume that the coefficient $t_{\infty}$ appearing in (\ref{decompoUnbounded}) satisfies $t_{\infty}\ne0$. Assume that both Problems (\ref{PbUnbounded}) and (\ref{PbUnboundedBis}) admits only the zero solution in $\mH^1(\om_{\infty})$ (absence of trapped modes for Problems (\ref{PbUnbounded}) and (\ref{PbUnboundedBis})). Then the following statements are valid:\\[5pt]
$i)$ (\textsc{non reflectivity}) There is an infinite sequence of values $1<L_1 <\dots <L_N <\dots$ such that for $L=L_n$, there holds $\rcoef=0$. Moreover, we have $\lim_{n\to+\infty}L_{n+1}-L_n=\pi/k$.\\[5pt]
$ii)$ (\textsc{perfect reflectivity}) There is an infinite sequence of values $1<\boldsymbol{L}_1 <\dots <\boldsymbol{L}_N <\dots$ such that for $L=\boldsymbol{L}_n$, there holds $\tcoef=0$. Moreover, we have $\lim_{n\to+\infty}\boldsymbol{L}_{n+1}-\boldsymbol{L}_n=\pi/k$.
\end{proposition}
\begin{proof}
We know that $\rcoef=(\rN+\rD)/2$ and $\tcoef=(\rN-\rD)/2$ (Formula (\ref{Formulas})). Moreover, for all $L\ge 1$, $\rN$ and $\rD$ are located on the unit circle $\mathbb{S}$ (\ref{NRJHalfguide}). The results of the previous section show that, as $L\to+\infty$, $\rD$ tends to a constant while $\rN$ runs continuously (and almost periodically) along $\mathbb{S}$ (here we use the assumptions of the proposition). From the intermediate value theorem, we deduce that there is an infinite sequence of values $1<L_1 <\dots <L_N <\dots$ such that for $L=L_n$, we have $r=-R$ and, therefore, $\rcoef=0$. This provides examples of geometries where there holds non reflectivity. As $n\to+\infty$, we have
\[
L_{n+1}-L_n=\pi/k+\dots,
\]
where the dots denote exponentially small terms. The proof of statement $ii)$ is similar. 
\end{proof}
\section{Non reflectivity and perfect reflectivity for larger vertical branches}\label{SectionTwoModes}
In the previous section, we explained how to exhibit geometries $\Om_L$ where we have $\rcoef=0$ or $\tcoef=0$ when $\ell\in(0;\pi/k)$ (we remind the reader that $\ell$ is the width of the vertical branch). Now we study the same question following the same approach when $\ell\in(\pi/k;2\pi/k)$.\\ 
\newline
When $\ell\in(\pi/k;2\pi/k)$, the main change compare to what has been done in Sections \ref{SectionAsymptotic}, \ref{SectionNonReflection} is that one propagating mode exists in the vertical branch of $\om_{\infty}$ for Problem (\ref{PbUnbounded}) with mixed boundary conditions. Set 
\[
w_{\bullet}^{\pm}(x,y)=(\alpha\ell/2)^{-1/2}e^{\pm i\alpha y}\sin(\pi x/\ell)\qquad \mbox{ with }\qquad\alpha:=\sqrt{k^2-(\pi/\ell)^2}.
\]
Problem (\ref{PbUnbounded}) admits the solutions
\begin{equation}\label{decompoUnboundedTwo}
\begin{array}{lcl}
U^-_{\infty} & = & \chi^-(w^++R_{\infty}\,w^-)+ \chi^{\circ}\,T_{\infty}\,w^+_{\bullet} + \tilde{U}^-_\infty,\\[3pt]
U^\bullet_{\infty} & = & \chi^-\,T^{\bullet}_{\infty}\,w^- +\chi^{\circ}(w^-_\bullet+R^{\bullet}_{\infty}\,w^+_\bullet) + \tilde{U}^\bullet_\infty,
\end{array}
\end{equation}
where $\tilde{U}^-_\infty$, $\tilde{U}^{\bullet}_\infty$ are functions in $\mH^1(\om_\infty)$. The scattering matrix
\begin{equation}\label{UnboundedScatteringMatrixBis}
\mathbb{S}_{\infty}:=\left(\begin{array}{cc}
R_{\infty} & T_{\infty} \\
T_{\infty}^{\bullet} & R_{\infty}^{\bullet} 
\end{array}\right)
\end{equation}
is unitary and symmetric (the proof is the same as the one of Proposition \ref{ProofUnitary} in Appendix). To obtain an asymptotic expansion of $\rD$ as $L$ goes to $+\infty$, we work exactly as in \S\ref{paragraphNeumann} where we derived an expansion for $\rN$. First, we compute an asymptotic expansion of $U$. For $U$, we make the ansatz 
\begin{equation}\label{AnsatzTwo}
U = U_{\infty}^- + A(L)\,U_{\infty}^{\bullet}+\dots
\end{equation}
where $A(L)$ is a gauge function, depending on $L$ but not on $(x,y)$, which has to be determined. On the segment $\Gamma_L=(-\ell/2;0)\times\{L\}$, we find 
\[
\begin{array}{lcl}
\partial_nU(x,L) & = & \partial_nU_{\infty}^-(x,L) + A(L)\,\partial_nU_{\infty}^{\bullet}(x,L)+\dots \\[3pt]
&=& i\alpha\,(\alpha\ell/2)^{-1/2}\,(\,T_{\infty}\,e^{i\alpha L}+A(L)\,(-e^{- i\alpha L}+R^{\bullet}_{\infty} \,e^{i\alpha L})\,)+\dots \ .
\end{array}
\]
Since $\partial_nU=0$ on $\Gamma_L$, we take 
\begin{equation}\label{gaugeFunctionTwo}
A(L)=\frac{-T_{\infty}}{-e^{- 2i\alpha L}+R^{\bullet}_{\infty}}.
\end{equation}
In order $A(L)$ to be defined for all $L>1$, we must have $|R^{\bullet}_{\infty}|\ne1$. Since $\mathbb{S}_{\infty}$ is unitary and symmetric, this is equivalent to have $T_{\infty}=T^{\bullet}_{\infty}\ne0$. When $T_{\infty}=0$, we can choose $A(L)=0$ and prove that $\rD = R_{\infty}+\dots$. When $T_{\infty}\ne0$, plugging expression (\ref{gaugeFunctionTwo}) in (\ref{AnsatzTwo}) and identifying the main contribution of the terms of each side of the equality as $x\to-\infty$ yields
\begin{equation}\label{formulaAsymptoticsTwo}
\rD = R_{\mrm{asy}}(L)+\dots\qquad\mbox{ with }\ R_{\mrm{asy}}(L):=R_{\infty}-\dsp\frac{(T^{\bullet}_{\infty})^2}{-e^{-2i\alpha L}+R^{\bullet}_{\infty}}.
\end{equation}
Working as in (\ref{eqnCenters})--(\ref{eqnRadius}), we can prove that the term $R_{\mrm{asy}}(L)$ runs along the unit circle as $L$ tends to $+\infty$.\\
\newline 
Coupling these results with the ones obtained in \S\ref{paragraphNeumann}, we deduce that when $\ell\in(\pi/k;2\pi/k)$, the scattering coefficients for Problem (\ref{PbChampTotal}) set in $\Om_L$ admit the asymptotic expansion
\begin{equation}\label{AsymExpanTwo}
\begin{array}{lll}
\rcoef = \rcoef_{\mrm{asy}}(L)+\dots\qquad\quad &\mbox{ with }&\qquad \rcoef_{\mrm{asy}}(L)=(r_{\mrm{asy}}(L) + R_{\mrm{asy}}(L))/2,\\[3pt]
\tcoef = \tcoef_{\mrm{asy}}(L)+\dots\qquad\quad &\mbox{ with }&\qquad\tcoef_{\mrm{asy}}(L)=(r_{\mrm{asy}}(L) - R_{\mrm{asy}}(L))/2.
\end{array}
\end{equation}
Here $r_{\mrm{asy}}(L)$, $R_{\mrm{asy}}(L)$ are respectively defined in (\ref{formulaAsymptotics}), (\ref{formulaAsymptoticsTwo}). In \S\ref{paragraphInitialPb}, where $\ell\in(0;\pi/k)$ so that  only one propagating mode exists in the vertical branch of $\Om_\infty$, we gave an explicit characterization of the sets $\{\rcoef_{\mrm{asy}}(L)\,|\,L\in(1;+\infty)\}$ and $\{\tcoef_{\mrm{asy}}(L)\,|\,L\in(1;+\infty)\}$. More precisely, we showed that they coincide with circles of radius $1/2$ passing through zero. In the present situation, this seems much less simple and numerical experiments in \S\ref{numerics2modes} show that the behaviour of $\rcoef_{\mrm{asy}}(L)$, $\tcoef_{\mrm{asy}}(L)$ when $L\to+\infty$ can be quite complicated. Let us just consider cases where there are $m,n\in\N^{\ast}:=\{1,2,\dots\}$, with $m>n$, such that 
\begin{equation}\label{RelationRatio}
k=\alpha\,\cfrac{m}{n} \qquad\Leftrightarrow \qquad \ell=\frac{\pi}{ k}\,\frac{m}{\sqrt{m^2-n^2}}\,.
\end{equation}
This boils down to assume that the width of the vertical branch $\ell$ is such that $k/\alpha$ is a rational number. Define $z=e^{-2i\alpha L/n}$. As $L\to+\infty$, $\rcoef_{\mrm{asy}}(L)$, $\tcoef_{\mrm{asy}}(L)$ run respectively along the sets 
\[
\begin{array}{lcl}
\mathcal{S}_{\rcoef}&:=&\Big\{\ \cfrac{1}{2}\,\Big(r_{\infty}-\dsp\frac{(t^{\circ}_{\infty})^2}{-z^m+r^{\circ}_{\infty}}\Big)- \frac{1}{2}\Big(R_{\infty}-\dsp\frac{(T^{\bullet}_{\infty})^2}{-z^n+R^{\bullet}_{\infty}}\Big)\,|\,z\in\mathbb{S}\ \Big\},\\[10pt]
\mathcal{S}_{\tcoef}&:=&\Big\{\ \cfrac{1}{2}\,\Big(r_{\infty}-\dsp\frac{(t^{\circ}_{\infty})^2}{-z^m+r^{\circ}_{\infty}}\Big)+ \frac{1}{2}\Big(R_{\infty}-\dsp\frac{(T^{\bullet}_{\infty})^2}{-z^n+R^{\bullet}_{\infty}}\Big)\,|\,z\in\mathbb{S}\ \Big\}.
\end{array}
\]
In other words, $\rcoef_{\mrm{asy}}(L)$, $\tcoef_{\mrm{asy}}(L)$ run $n\pi/\alpha$-periodically along the close curves $\mathcal{S}_{\rcoef}$, $\mathcal{S}_{\tcoef}$ in the complex plane. Moreover, for any $L_{\star}>1$, for $L\in[L_{\star};L_{\star}+n\pi/\alpha]$, $L\mapsto r_{\mrm{asy}}(L)$ (resp. $L\mapsto R_{\mrm{asy}}(L)$) runs continuously $m$ times (resp. $n$ times) along $\mathbb{S}$. Therefore, according to the intermediate value theorem, we know that there exist at least $m-n$ values of $L\in [L_{\star};L_{\star}+n\pi/\alpha]$ such that $R_{\mrm{asy}}(L)=r_{\mrm{asy}}(L)$ and $m-n$ other values of $L\in [L_{\star};L_{\star}+n\pi/\alpha]$ such that $R_{\mrm{asy}}(L)=-r_{\mrm{asy}}(L)$. Since $\rcoef=(r_{\mrm{asy}}(L)+R_{\mrm{asy}}(L))/2+\dots$, and $\tcoef=(r_{\mrm{asy}}(L)-R_{\mrm{asy}}(L))/2+\dots$, we infer that there are some constants $\alpha(L_{\star})\le0$ and $\beta(L_{\star})\ge0$ (exponentially small with respect to $L_{\star}$) such that $L\mapsto \rcoef$ and $L\mapsto \tcoef$ vanish at least $m-n$ times in $[L_{\star}+\alpha(L_{\star});L_{\star}+n\pi/\alpha+\beta(L_{\star})]$. This provides examples of geometries where we have non reflectivity or perfect reflectivity with $\ell\in(\pi/k;2\pi/k)$.\\
\newline
When $k/\alpha$ is not a rational number, since $r_{\mrm{asy}}(L)$ runs faster than $R_{\mrm{asy}}(L)$ along the unit disk  (because $k>\alpha$), we can still conclude that there are some $L$ such that $\rcoef=0$ or $\tcoef=0$. However, in general the values of $L$ such that  $\rcoef=0$ or $\tcoef=0$ do not form an (approximately) periodic sequence. We summarize these results in the following proposition.
\begin{proposition}\label{WideBranch}
Assume that one of the coefficients $t_{\infty}$, $T_{\infty}$ appearing in (\ref{decompoUnbounded}), (\ref{decompoUnboundedTwo}) satisfies $t_{\infty}\ne0$ or $T_{\infty}\ne0$. Assume that both Problems (\ref{PbUnbounded}) and (\ref{PbUnboundedBis}) admits only the zero solution in $\mH^1(\om_{\infty})$ (absence of trapped modes for Problems (\ref{PbUnbounded}) and (\ref{PbUnboundedBis})). Then the following statements are valid:\\[5pt]
$i)$ (\textsc{non reflectivity}) There is an infinite sequence of values $1<L_1 <\dots <L_N <\dots$ such that for $L=L_n$, there holds $\rcoef=0$.\\[5pt]
$ii)$ (\textsc{perfect reflectivity}) There is an infinite sequence of values $1<\boldsymbol{L}_1 <\dots <\boldsymbol{L}_N <\dots$ such that for $L=\boldsymbol{L}_n$, we have $\tcoef=0$.
\end{proposition}

\section{Perfect invisibility}\label{SectionCompleteInvisibility}
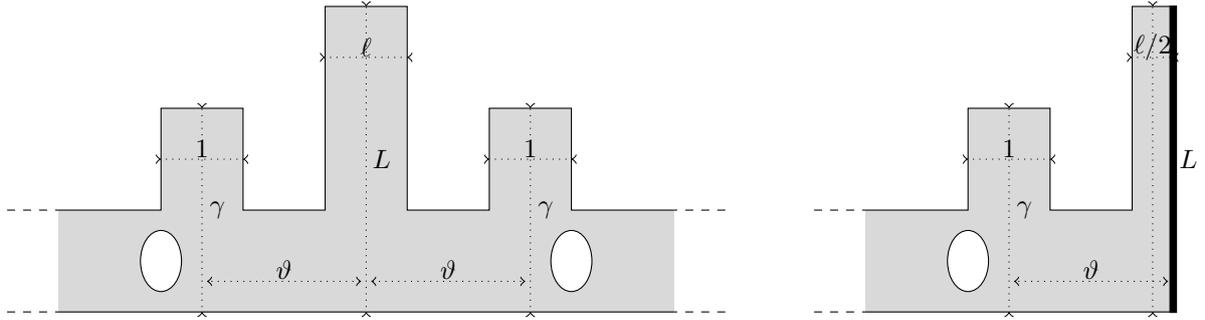
\begin{figure}[!ht]
\centering
\begin{tikzpicture}[scale=1.35]
\draw[fill=gray!30,draw=none](-3,0) rectangle (3,1);
\draw[fill=gray!30,draw=none](-0.4,0) rectangle (0.4,3);
\draw[fill=gray!30,draw=none](-2,0) rectangle (-1.2,2);
\draw[fill=gray!30,draw=none](2,0) rectangle (1.2,2);
\draw (-3,1)--(-2,1)--(-2,2)--(-1.2,2)--(-1.2,1)--(-0.4,1)--(-0.4,3)--(0.4,3)--(0.4,1)--(1.2,1)--(1.2,2)--(2,2)--(2,1)--(3,1);
\draw (3,0)--(-3,0);
\draw[dashed] (-3.5,1)--(-3,1); 
\draw[dashed] (-3.5,0)--(-3,0); 
\draw[dashed] (3.5,1)--(3,1); 
\draw[dashed] (3.5,0)--(3,0); 
\draw[dotted,>-<] (-0.45,2.5)--(0.45,2.5);
\draw[dotted,>-<] (0,-0.05)--(0,3.05);
\draw[dotted,>-<] (-2.05,1.5)--(-1.15,1.5);
\draw[dotted,>-<] (2.05,1.5)--(1.15,1.5);
\draw[dotted,>-<] (-1.6,-0.05)--(-1.6,2.05);
\draw[dotted,>-<] (1.6,-0.05)--(1.6,2.05);
\draw[dotted,<->] (0.05,0.3)--(1.55,0.3);
\draw[dotted,<->] (-0.05,0.3)--(-1.55,0.3);
\node at (0,2.6){\small $\ell$};
\node at (1.6,1.6){\small $1$};
\node at (-1.6,1.6){\small $1$};
\node at (0.15,1.5){\small $L$};
\node at (1.75,1){\small $\gamma$};
\node at (-1.45,1){\small $\gamma$};
\node at (0.8,0.42){\small $\vartheta$};
\node at (-0.8,0.42){\small $\vartheta$};
\draw[fill=white] (-2,0.5) ellipse (0.2 and 0.3);
\draw[fill=white] (2,0.5) ellipse (0.2 and 0.3);
\end{tikzpicture}\qquad\quad \begin{tikzpicture}[scale=1.35]
\draw[fill=gray!30,draw=none](-3,0) rectangle (0,1);
\draw[fill=gray!30,draw=none](-0.4,0) rectangle (0,3);
\draw[fill=gray!30,draw=none](-2,0) rectangle (-1.2,2);
\draw (-3,1)--(-2,1)--(-2,2)--(-1.2,2)--(-1.2,1)--(-0.4,1)--(-0.4,3)--(0,3)--(0,0)--(-3,0);
\draw[line width=1mm] (0,-0.01)--(0,3.01);
\draw[dashed] (-3.5,1)--(-3,1); 
\draw[dashed] (-3.5,0)--(-3,0); 
\draw[dotted,>-<] (-0.45,2.5)--(0.05,2.5);
\draw[dotted,>-<] (-0.2,-0.05)--(-0.2,3.05);
\draw[dotted,>-<] (-2.05,1.5)--(-1.15,1.5);
\draw[dotted,>-<] (-1.6,-0.05)--(-1.6,2.05);
\draw[dotted,<->] (-0.05,0.3)--(-1.55,0.3);
\node at (-0.2,2.6){\small $\ell/2$};
\node at (-1.6,1.6){\small $1$};
\node at (0.15,1.5){\small $L$};
\node at (-1.45,1){\small $\gamma$};
\node at (-0.8,0.42){\small $\vartheta$};
\draw[fill=white] (-2,0.5) ellipse (0.2 and 0.3);
\end{tikzpicture}
\caption{Domains $\Om_{L}^\gamma$ (left) and $\om_{L}^\gamma$ (right).\label{TwoChem}} 
\end{figure}

Up to now, we have explained how to find geometries where $\rcoef=0$ (non reflectivity) or $\tcoef=0$ (perfect reflectivity). In this section, we explain how to get $\tcoef=1$ (perfect invisibility). Since $\tcoef=(\rN-\rD)/2$ (Formula (\ref{Formulas})), we must impose both $\rN=1$ and $\rD=-1$. To proceed, we  work in a new geometry $\Om_{L}^\gamma$ (see Figure \ref{TwoChem}, left) which coincides with the region 
\[
\{ (x,y)\in\R\times(0;1)\ \cup\  (-\ell/2;\ell/2)\times [1;L)\ \cup\ (\pm\vartheta-1/2;\pm\vartheta+1/2)\times [1;\gamma)
\]
outside a bounded domain. Again we assume that $\Om_{L}^\gamma$ is symmetric with respect to the $(Oy)$ axis ($\Om_{L}^\gamma=\{(-x,y)\,|\,(x,y)\in\Om_{L}^\gamma\}$), connected and that its boundary is Lipschitz. Here $\gamma>1$ and $\vartheta>\ell/2+0.5$. The parameter $\vartheta$ is chosen only so that the central branch is distinct from the two others. Let $\om_{L}^\gamma$ refer to the half-waveguide such that  $\om_{L}^\gamma:=\{(x,y)\in\Om_{L}^\gamma\,|\,x<0\}$ (Figure \ref{TwoChem} right). Again, denote $\rcoef$, $\tcoef$ (resp. $\rN$, $\rD$) the scattering coefficients for Problem (\ref{PbChampTotal}) (resp. for Problems (\ref{PbChampTotalSym}), (\ref{PbChampTotalAntiSym})) set in $\Om_L^{\gamma}$ (resp. $\om_L^{\gamma}$).\\
\newline
For $\ell\in(0;\pi/k)$ and a given $\gamma>0$, as explained in \S\ref{paragraphMixed}, $\rD$ tends to a constant $R_{\infty}$ located on the unit circle $\mathbb{S}$. Making $\gamma\to+\infty$, we can prove as in \S\ref{paragraphNeumann} that $R_{\infty}$ runs continuously along $\mathbb{S}$.  This allows one to deduce that there is $\gamma=\gamma_{\infty}$ such that $R_{\infty}=-1$. Then, tuning $\gamma$ into $\gamma(L)$, with $\gamma(L)$ exponentially close to $\gamma_{\infty}$, we can impose $\rD=-1$ for all $L$ sufficiently large. On the other hand, the coefficient $\rN$ runs along $\mathbb{S}$ as $L\to+\infty$. Therefore, almost periodically, we have $\rN=1$ and $\rD=-1$ so that $\tcoef=1$.\\
\newline
The complete rigorous justification of this approach is rather intricate. Therefore we do not formulate a proposition with precise assumptions (which would look like the ones of Proposition \ref{WideBranch}). However we will see in \S\ref{NumericsPerfectInvisibility} that numerically this methodology seems efficient to create waveguides where $\tcoef=1$.

\section{Numerical results}\label{SectionNumerics}
We give here illustrations of the results obtained in the previous section. For given $\ell>0$, $L>1$, we approximate numerically the solution of Problem (\ref{PbChampTotal}) with a P2 finite element method set in the bounded domain $\Om_{b}:=\{ (x,y)\in(-8;8)\times(0;1)\ \cup\  (-\ell/2;\ell/2)\times [1;L)\}$ (see Figure \ref{figResult1}). We emphasize that we work in a very simple geometry but other waveguides, for example with voids as depicted in Figure \ref{DomainOriginal2D}, can be considered. In particular in this geometry, one can use analytic methods (see e.g. \cite{FeEv98}) instead of finite element techniques. At $x=\pm 8$, a Dirichlet-to-Neumann map with 20 modes serves as a transparent boundary condition. From the numerical solution $v_h$, we deduce approximations $\rcoef_h$, $\tcoef_h$ of the scattering coefficients $\rcoef$, $\tcoef$ defined in (\ref{defZetaL}) (here $h$ refers to the mesh size). Then, we display the behaviour of $\rcoef_h$, $\tcoef_h$ with respect to $L$.  For the numerics, the wavenumber $k$ is set to $k=0.8\pi \in(0;\pi)$.

\subsection{Case 1: one propagating mode exists in the vertical branch of $\Om_{\infty}$}
First, we investigate the situation of \S\ref{paragraphInitialPb} where $\ell\in(0;\pi/k)$. To obtain the results of Figures \ref{CoeffS1}-\ref{VerifOscillations1}, we take $\ell=1\in(0;\pi/k)$. In Figure  \ref{CoeffS1}, we observe that, asymptotically as $L\to+\infty$, the coefficients $\rcoef_h$, $\tcoef_h$ run along circles. This is coherent with what was derived in (\ref{CircleFirstCase}). Figure \ref{VerifOscillations1} confirms that the coefficients $\rcoef$, $\tcoef$ are asymptotically periodic with respect to $L\to+\infty$. More precisely, in (\ref{CircleFirstCase}), we found that the period must be equal to $\pi/k=1.25 $, which is more or less what is obtained in Figure \ref{VerifOscillations1}. Figure \ref{VerifOscillations1} also confirms that, periodically, $\rcoef$, $\tcoef$ are equal to zero.

\begin{figure}[!ht]
\centering
\includegraphics[width=0.42\textwidth]{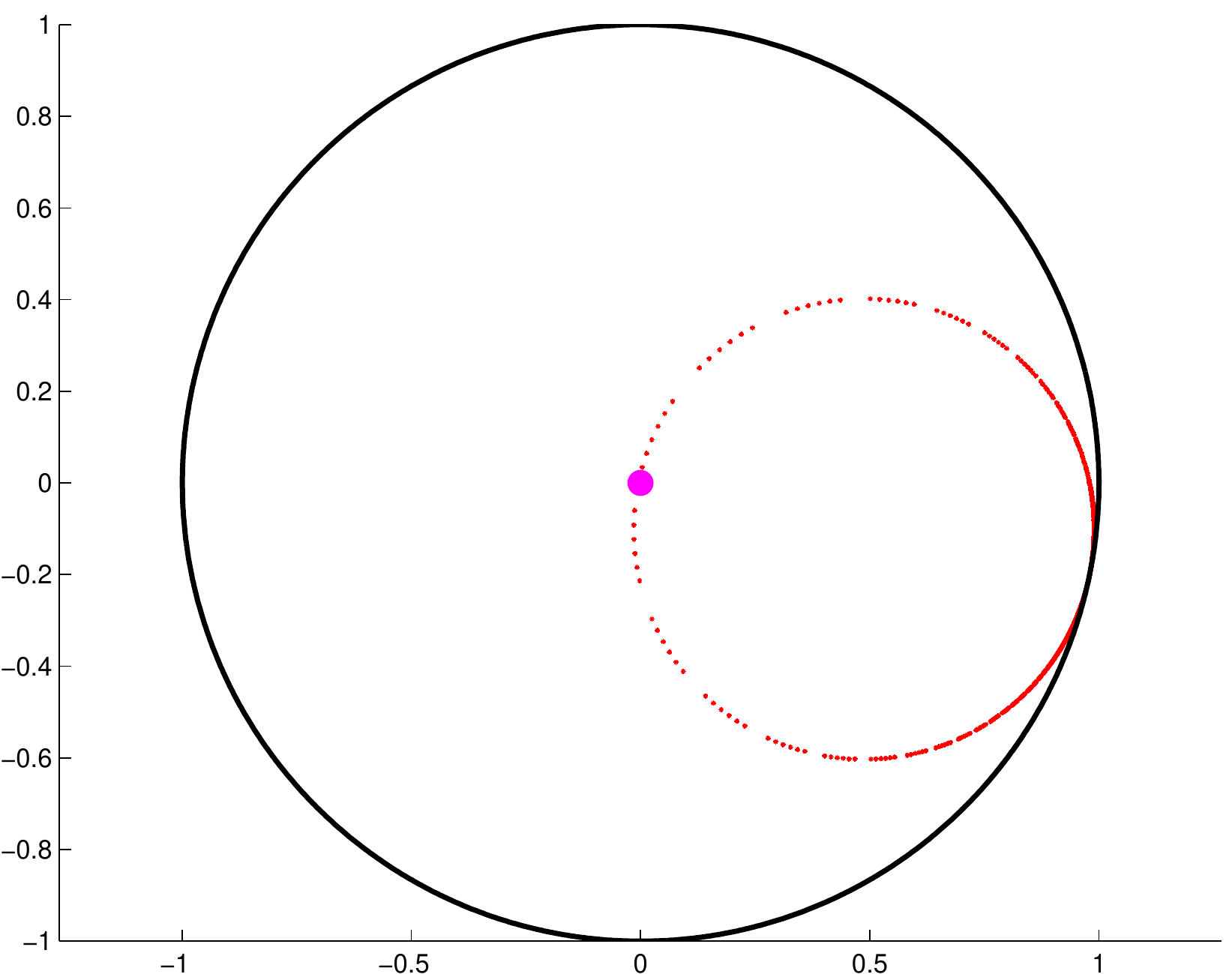}\quad\includegraphics[width=0.42\textwidth]{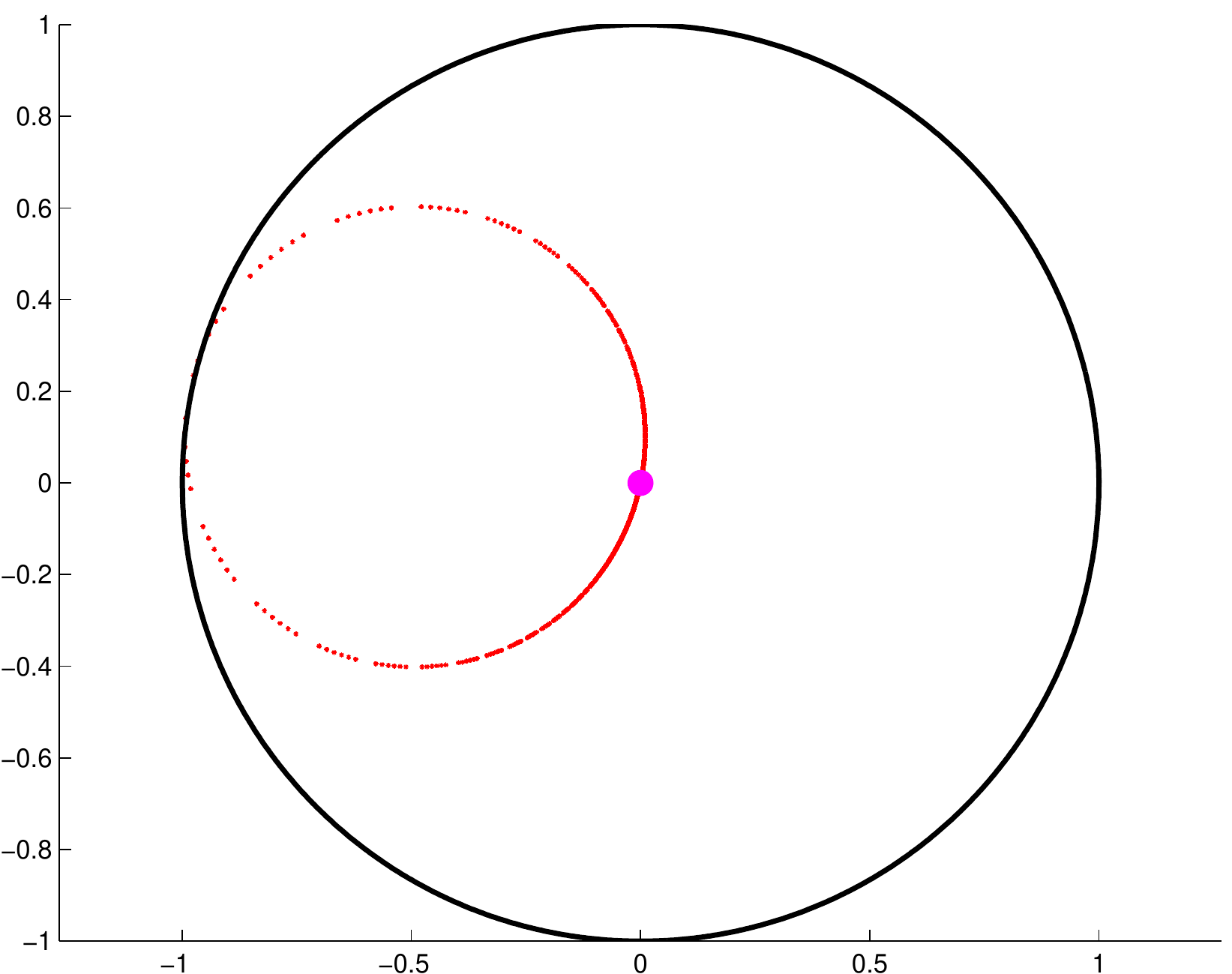}
\caption{Coefficients $\tcoef_h$ (left) and $\rcoef_h$ (right) for $\ell=1\in(0;\pi/k)$ and $L\in(2;10)$. Note that due to the conservation of energy, the coefficients $\rcoef$, $\tcoef$ are located inside the unit disk.\label{CoeffS1}}
\end{figure}

\begin{figure}[!ht]
\centering
\includegraphics[width=0.42\textwidth]{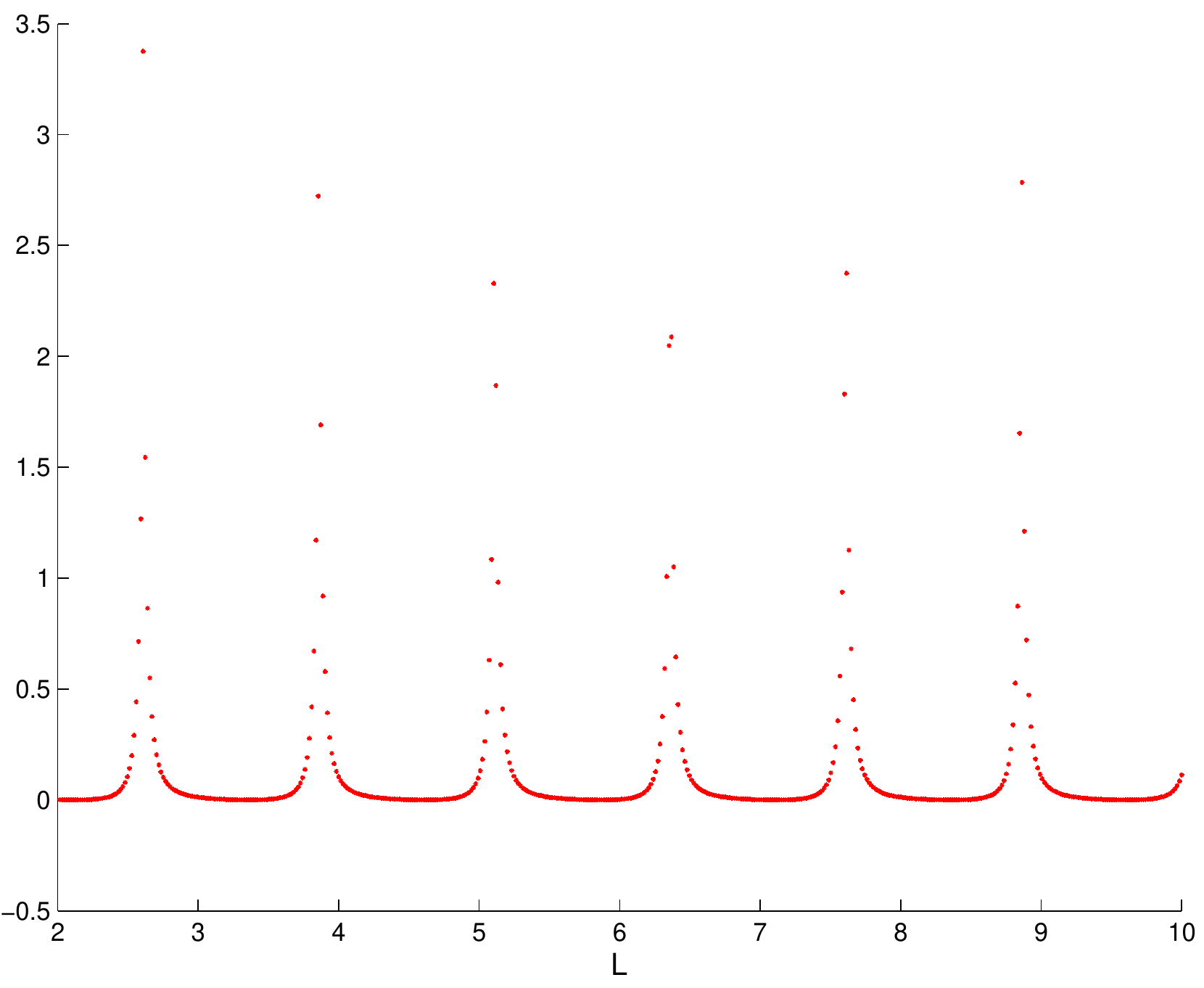}\quad\includegraphics[width=0.42\textwidth]{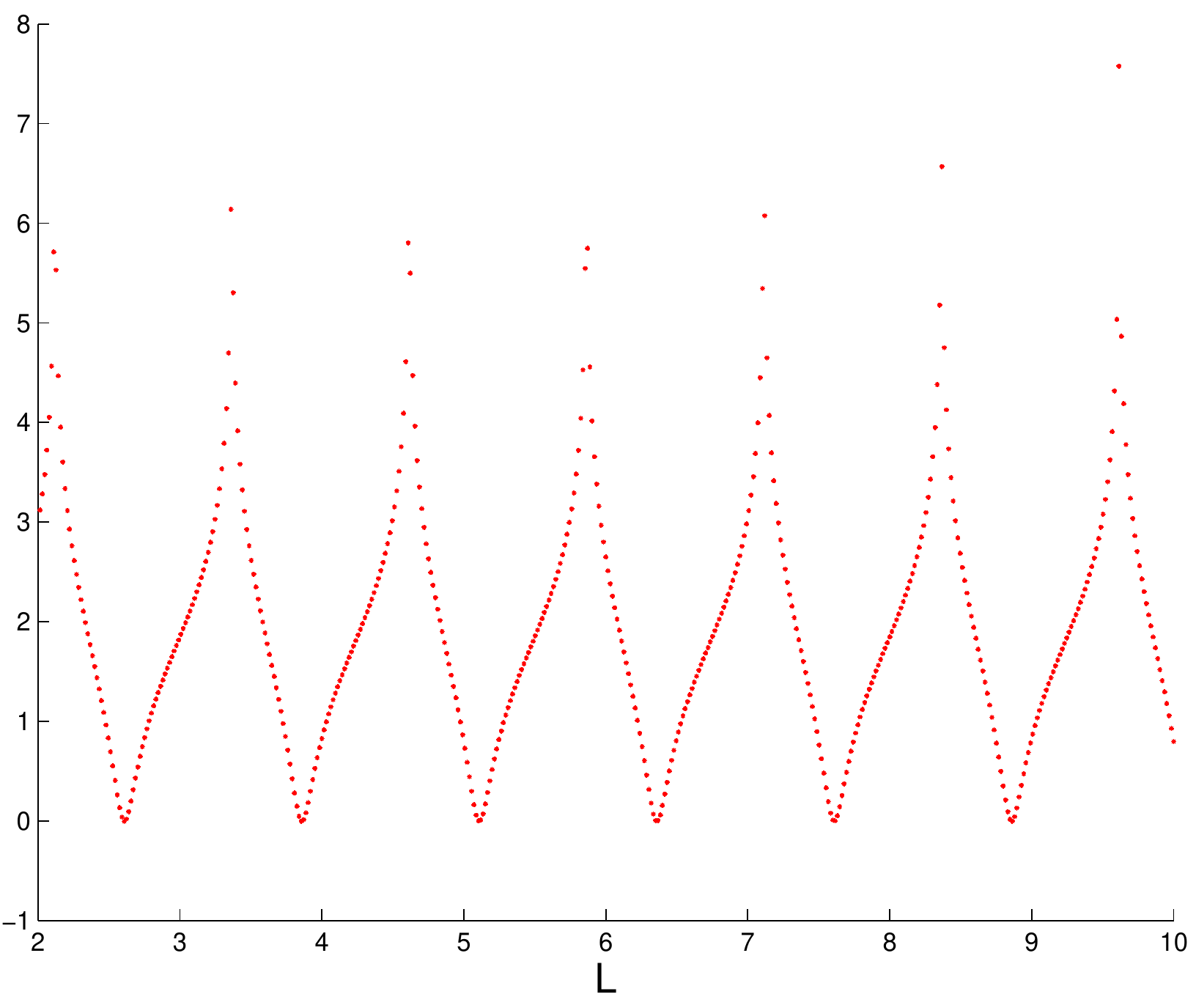}
\caption{Curves $L\mapsto -\ln|\tcoef_{h}|$ (left) and $L\mapsto -\ln|\rcoef_{h}|$ (left) for $\ell=1$.\label{VerifOscillations1}}
\end{figure}

\noindent In the next series of experiments, we study the properties of the  asymptotic circles $\{\tcoef_{\mrm{asy}}(L)\,|\,L\in(1;+\infty)\}$ and $\{\rcoef_{\mrm{asy}}(L)\,|\,L\in(1;+\infty)\}$  defined in (\ref{CircleFirstCase}) with respect to the width $\ell\in(0;\pi/k)$ of the vertical branch. For each $\ell\in(0;\pi/k)$, we showed that $\{\tcoef_{\mrm{asy}}(L)\,|\,L\in(1;+\infty)\}$ (resp. $\{\rcoef_{\mrm{asy}}(L)\,|\,L\in(1;+\infty)\}$) is a circle of radius $1/2$ centered at $- R_{\infty}/2$ (resp. $R_{\infty}/2$). Therefore, numerically it suffices, for all $\ell\in(0;\pi/k)$, to compute an approximation of the coefficient $R_{\infty}$ solving Problem (\ref{PbUnbounded}) set in $\om_{\infty}$. The results are displayed in Figure \ref{Circlep}. If we take $\ell=1$, we observe that the obtained circles coincide with the ones of Figure \ref{CoeffS1}.~

\begin{figure}[!ht]
\centering
\includegraphics[width=0.42\textwidth]{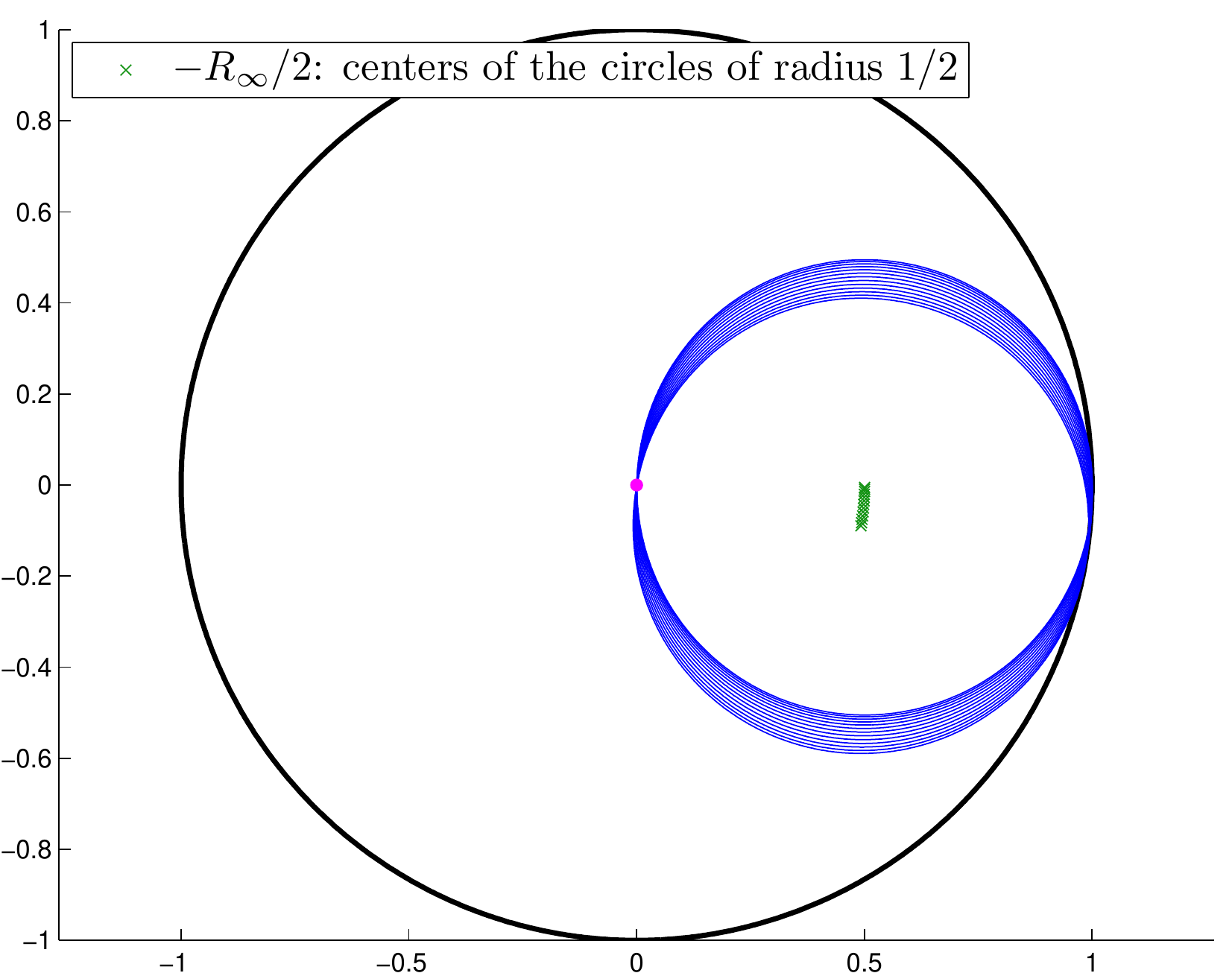}\quad\includegraphics[width=0.42\textwidth]{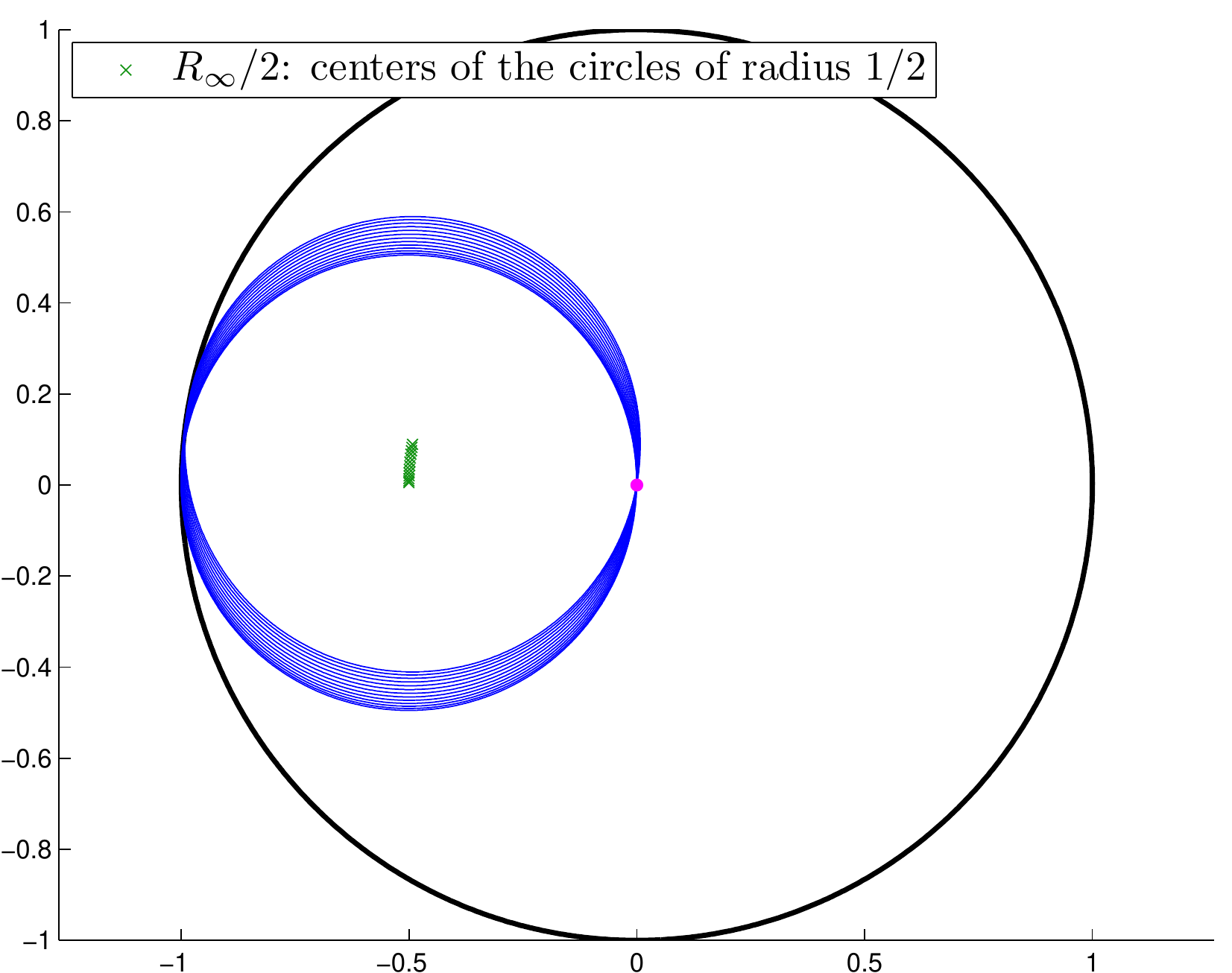}
\caption{Asymptotic circles with respect to $\ell\in(0;\pi/k)$. For each $\ell$, there is one circle. Left: $\{ \tcoef_{\mrm{asy},\,h}(L)\,|\,L\in(1;+\infty)\}$, right: $\{ \rcoef_{\mrm{asy},\,h}(L)\,|\,L\in(1;+\infty)\}$. \label{Circlep}}
\end{figure}

\subsection{Case 2: two propagating modes exist in the vertical branch of $\Om_{\infty}$}\label{numerics2modes}

Now, we consider the case $\ell\in(\pi/k;2\pi/k)$ which has been studied in  Section \ref{SectionTwoModes}. In Figures \ref{ApproxTwop} left and \ref{ApproxTwom} left, we display the behaviour of $\tcoef_{h}$, $\rcoef_{h}$ for $\ell=1.4$ and $L\in(2;10)$. Independently, numerically we can compute the coefficients $r_{\infty}$, $t_{\infty}$, $r^{\circ}_{\infty}$ (resp. $R_{\infty}$, $T_{\infty}$, $R^{\bullet}_{\infty}$) appearing in (\ref{decompoUnbounded}) (resp. (\ref{decompoUnboundedTwo})). Hence, we can approximate the coefficients $\rcoef_{\mrm{asy}}(L)$, $\tcoef_{\mrm{asy}}(L)$ defined in (\ref{AsymExpanTwo}). We denote $\rcoef_{\mrm{asy},\,h}(L)$, $\tcoef_{\mrm{asy},\,h}(L)$ these approximations. The results are given in Figures \ref{ApproxTwop} right and \ref{ApproxTwom} right. We observe that the curves are in good agreement, that is $\tcoef_{h}$ (resp. $\rcoef_{h}$) and $\tcoef_{\mrm{asy},\,h}(L)$ (resp. $\rcoef_{\mrm{asy},\,h}(L)$) are close to each other. Figure \ref{CurvesErreur}, where the errors $L\mapsto |\tcoef_{h}-\tcoef_{\mrm{asy},\,h}(L)|$ and $L\mapsto |\rcoef_{h}-\rcoef_{\mrm{asy},\,h}(L)|$ are displayed, confirms this impression. Errors are small even though $L$ is not that large. This is due to exponential convergence with respect to $L$. Actually on Figure \ref{CurvesErreur}, we observe that rapidly the numerical error becomes predominant with respect to the asymptotic error as $L$ increases.

\begin{figure}[!ht]
\centering
\includegraphics[width=0.45\textwidth]{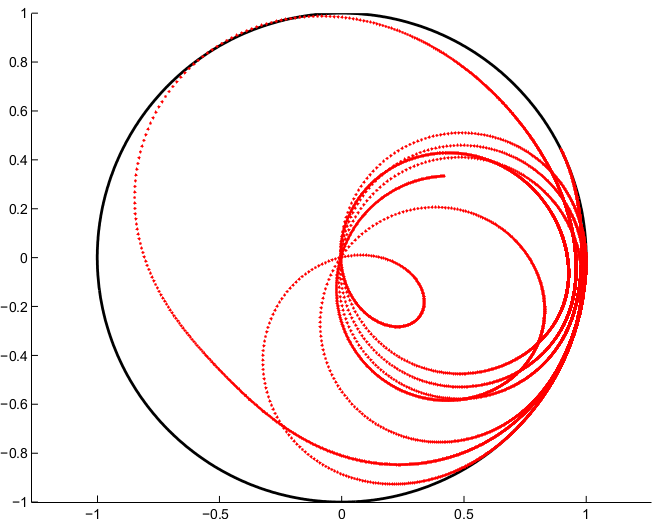}\quad\includegraphics[width=0.45\textwidth]{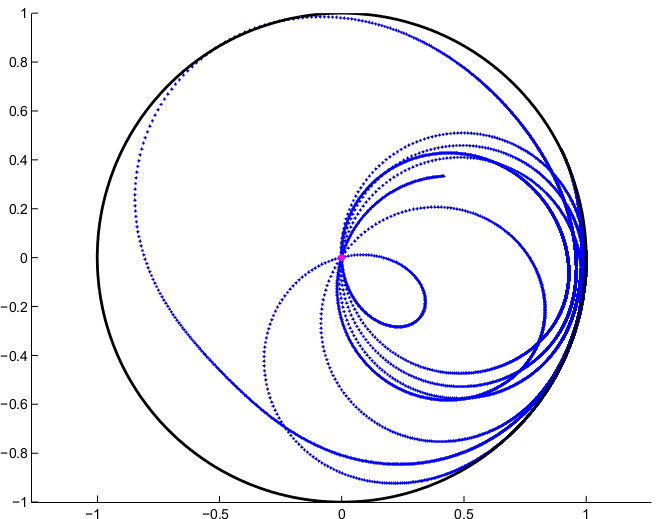}
\caption{Coefficients $\tcoef_{L,\,h}$ (left) and $\tcoef_{\mrm{asy},\,h}(L)$ (right) for $\ell=1.4\in(\pi/k;2\pi/k)$ and $L\in(2;10)$.\label{ApproxTwop}}
\end{figure}

\begin{figure}[!ht]
\centering
\includegraphics[width=0.45\textwidth]{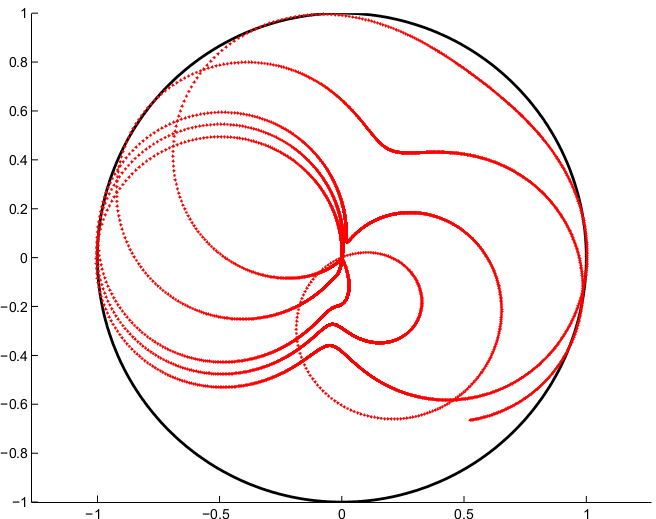}\quad\includegraphics[width=0.45\textwidth]{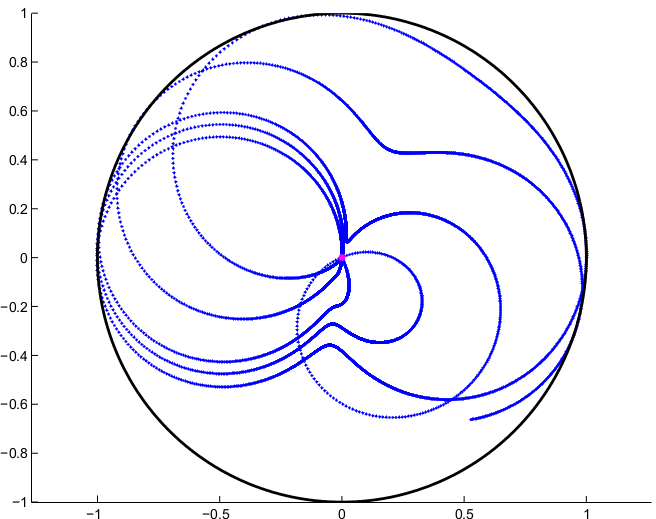}
\caption{Coefficients $\rcoef_{L,\,h}$ (left) and $\rcoef_{\mrm{asy},\,h}(L)$ (right) for $\ell=1.4\in(\pi/k;2\pi/k)$ and $L\in(2;10)$.\label{ApproxTwom}}
\end{figure}

\begin{figure}[!ht]
\centering
\includegraphics[scale=0.8,trim={0cm 2.6cm 0cm 3cm},clip]{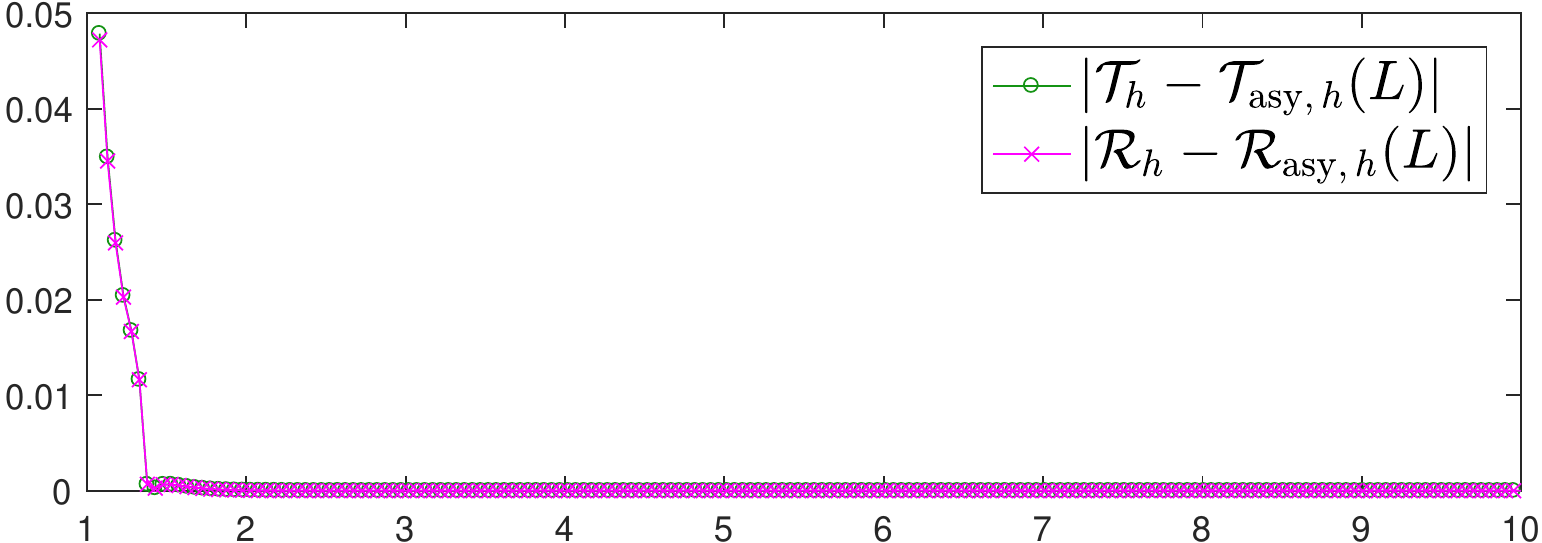}\caption{Curves $L\mapsto |\tcoef_{h}-\tcoef_{\mrm{asy},\,h}(L)|$ and $L\mapsto |\rcoef_{h}-\rcoef_{\mrm{asy},\,h}(L)|$ for $\ell=1.4$.\label{CurvesErreur}}
\end{figure}~

\newpage

\noindent In Figure \ref{Album}, we display the behaviour of the curves $\{\tcoef_{\mrm{asy}}(L)\,|\,L\in(1;+\infty)\}$ and $\{\rcoef_{\mrm{asy}}(L)\,|\,L\in(1;+\infty)\}$ for several particular values of the width $\ell$ of the vertical branch of the waveguide. More precisely, we choose $\ell$ such that 
\[
k=m\alpha\qquad\Leftrightarrow \qquad \ell=\frac{\pi}{ k}\,\frac{m}{\sqrt{m^2-n^2}}. 
\]
with $m=2,3,4,5$. In (\ref{RelationRatio}), we showed that in this case, $\{\tcoef_{\mrm{asy}}(L)\,|\,L\in(1;+\infty)\}$ and $\{\rcoef_{\mrm{asy}}(L)\,|\,L\in(1;+\infty)\}$ must be close curves in the complex plane. Our simulations are in accordance with this result. 

\newpage

\begin{figure}[!ht]
\centering
\begin{tabular}{cc} 
\includegraphics[width=0.43\textwidth]{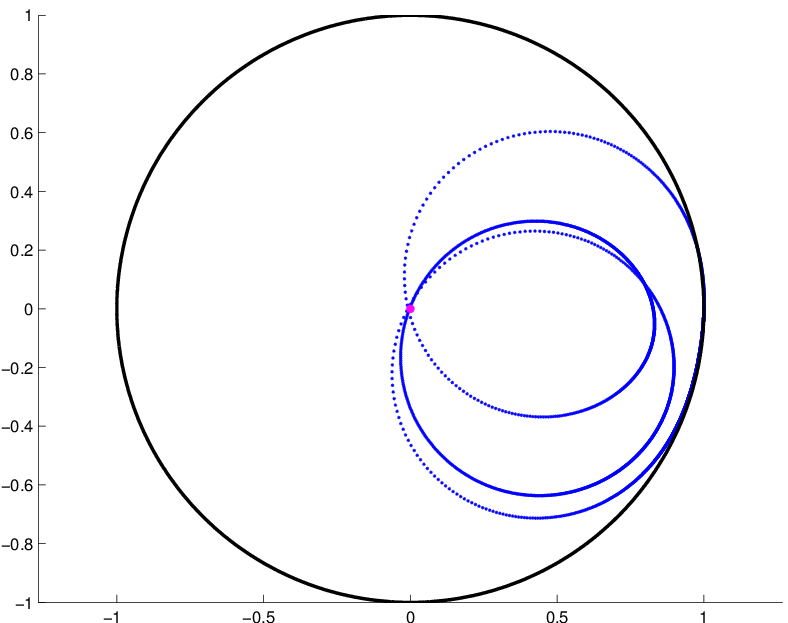} & \includegraphics[width=0.43\textwidth]{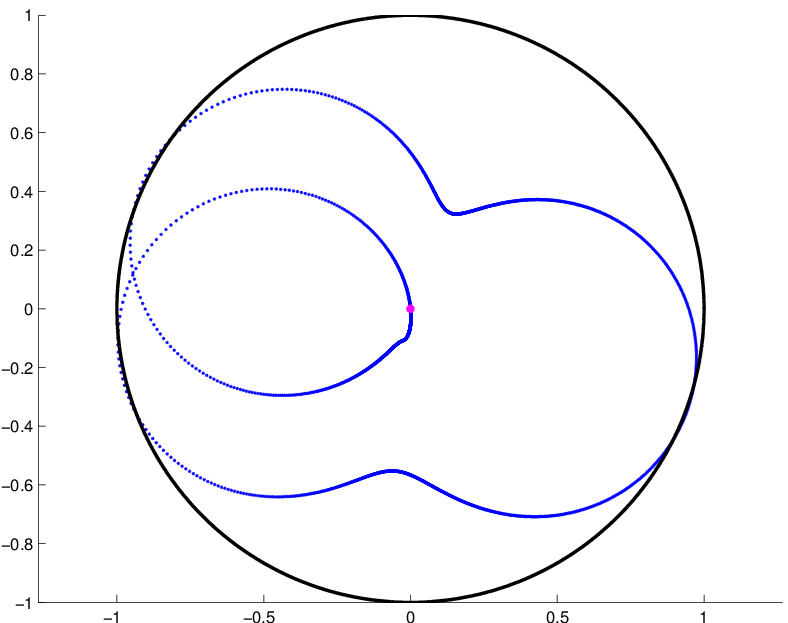}\\[4pt]
\includegraphics[width=0.43\textwidth]{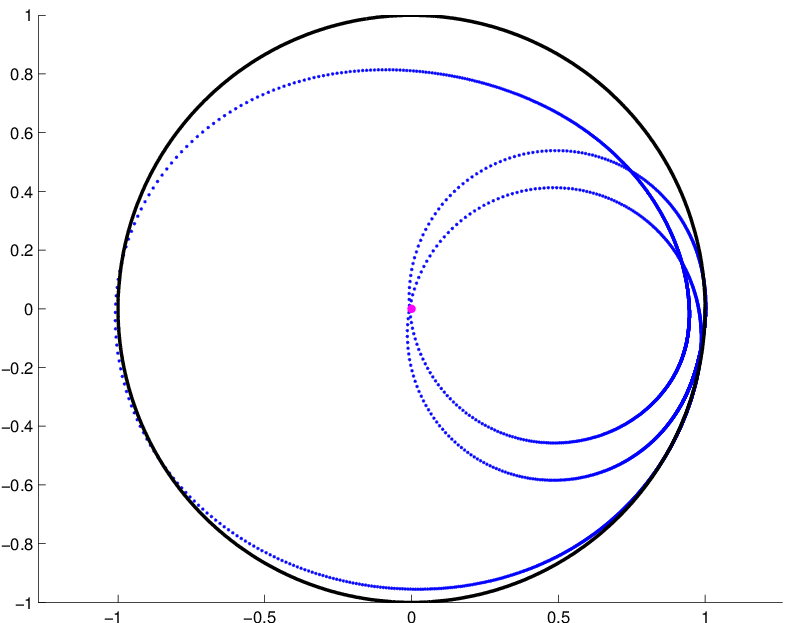} & \includegraphics[width=0.43\textwidth]{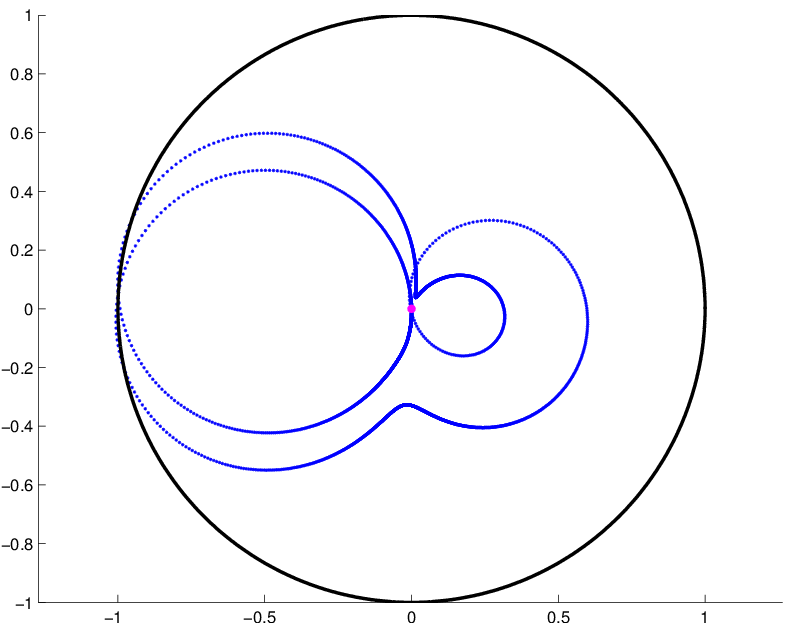}\\[4pt]
\includegraphics[width=0.43\textwidth]{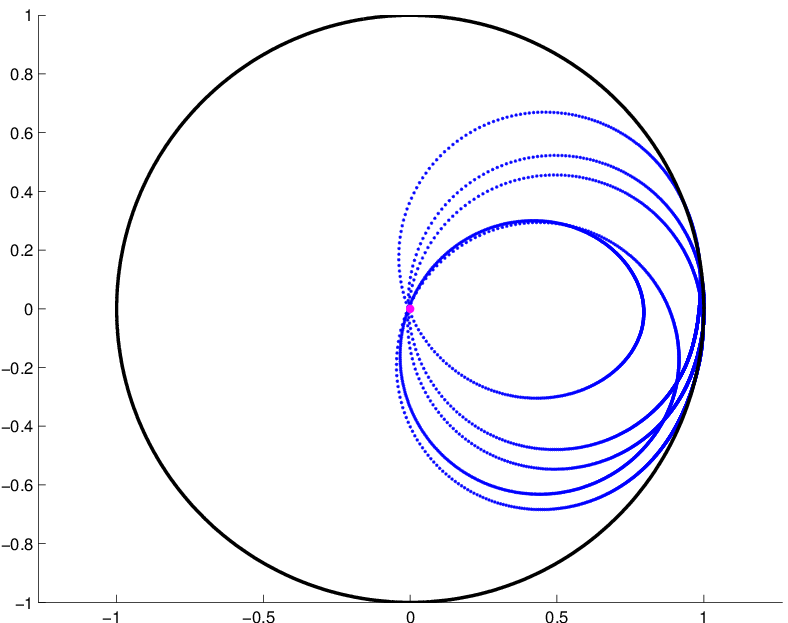} & \includegraphics[width=0.43\textwidth]{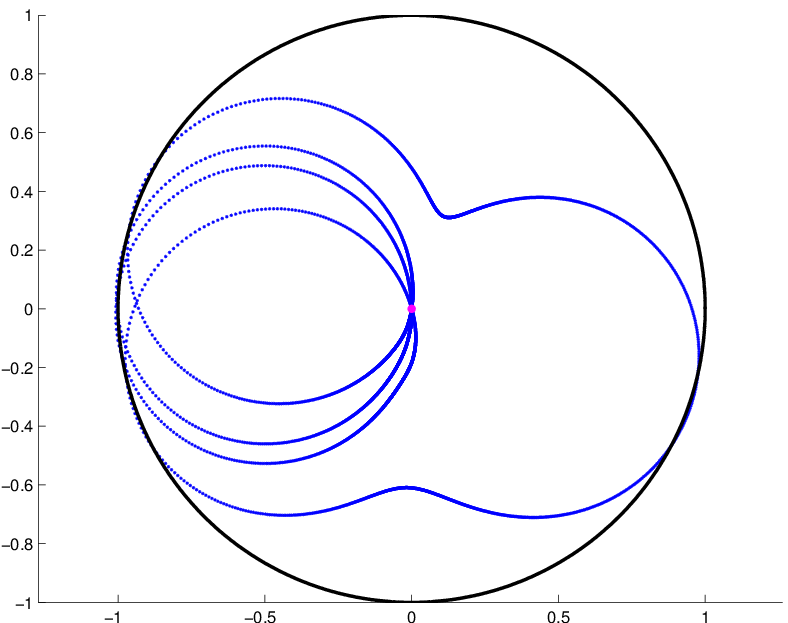}\\[4pt]
\includegraphics[width=0.43\textwidth]{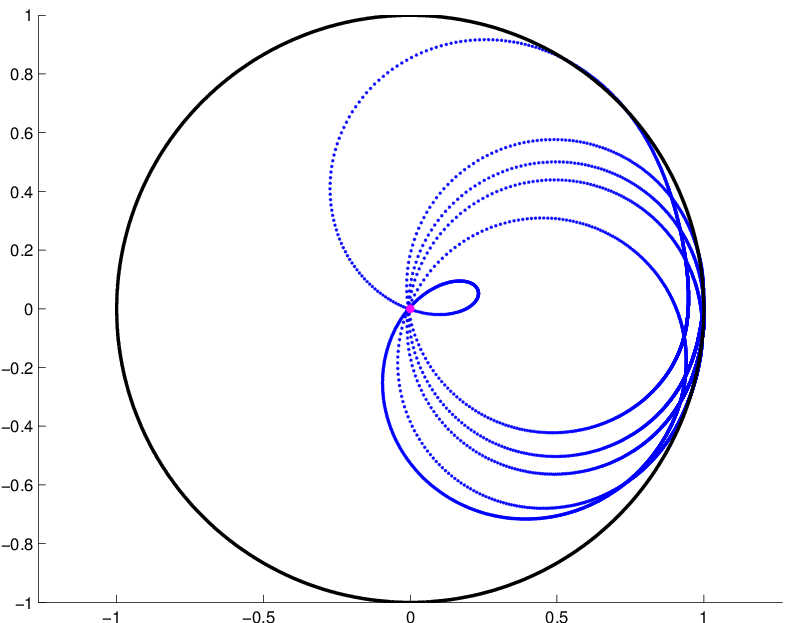} & \includegraphics[width=0.43\textwidth]{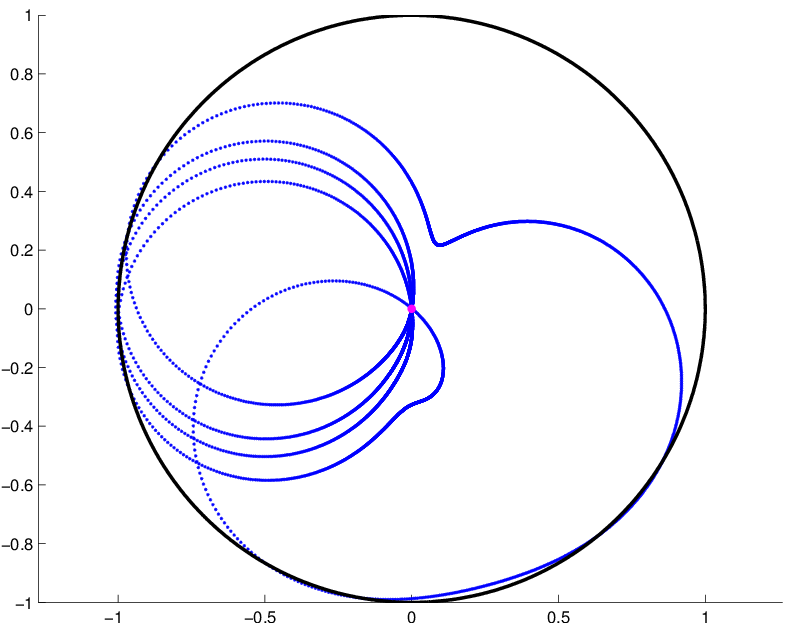}
\end{tabular}
\caption{Curves $\{\tcoef_{\mrm{asy},\,h}(L)\,|\,L\in(1;+\infty)\}$ (left) and $\{\rcoef_{\mrm{asy},\,h}(L)\,|\,L\in(1;+\infty)\}$ (right) for several values of $\ell\in(\pi/k;2\pi/k)$. For the line $m-1$, $m=2,3,4,5$, we take $\ell=m\pi/(k\sqrt{m^2-1})$. \label{Album}}
\end{figure}

\newpage
\clearpage

\noindent In Figure \ref{ApproxTwomDense}, we represent the numerical approximation of the curves $\{\tcoef_{\mrm{asy}}(L)\,|\,L\in(2;200)\}$ and $\{\rcoef_{\mrm{asy}}(L)\,|\,L\in(2;200)\}$ for  $\ell=1.7$. We can prove in this case that the ratio $k/\alpha$ is an irrational number. As predicted, the curves goes through zero. It seems also that they fill the unit disk. However, we are not able to prove it. 

\begin{figure}[!ht]
\centering
\includegraphics[scale=0.47]{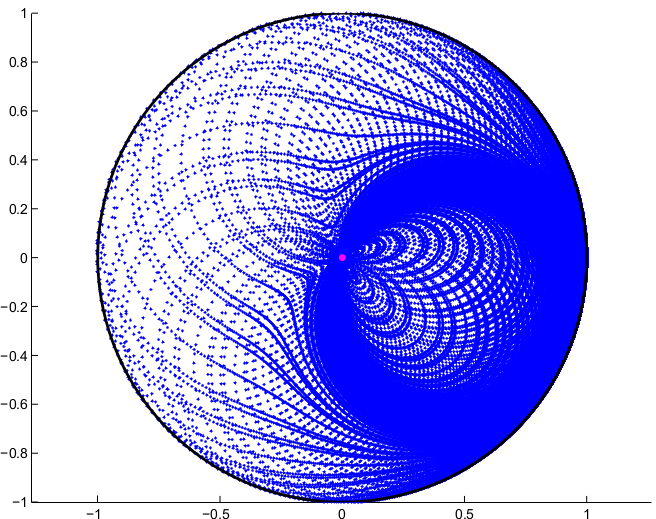}\quad\includegraphics[scale=0.47]{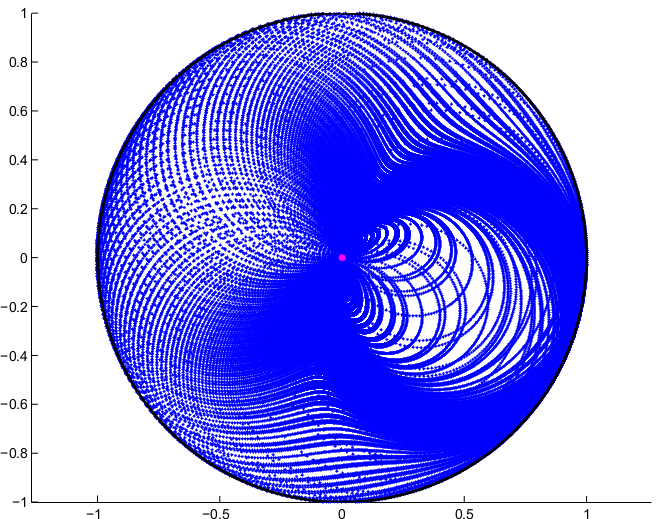}
\caption{Coefficients $\tcoef_{\mrm{asy},\,h}(L)$ (left) and $\rcoef_{\mrm{asy},\,h}(L)$ (right) for $\ell=1.7\in(\pi/k;2\pi/k)$ and $L\in(2;200)$.\label{ApproxTwomDense}}
\end{figure}

\subsection{Non reflectivity}
We give examples of waveguides where $\rcoef=0$ for well-chosen $\ell$ and $L$. Numerically we set $\ell$ and then we compute $\rcoef_{h}$ for a range of $L$. Finally, we select the $L$ such that $-\ln|\rcoef_{h}|$ is maximum.  In Figure \ref{figResult1} top, the parameters are set to $\ell=1\in(0;\pi/k)$ and $L=3.3649$ ($\rcoef_{h}\approx (-7.8+8.9i).10^{-6}$). In Figure \ref{figResult1} bottom, we have $\ell=2\in(\pi/k;2\pi/k)$ and $L=5.5329$ ($\rcoef_{h}\approx (1.4+i).10^{-5}$). As expected, the amplitude of the field $v_h-w^+_h$ is very small in the input lead. In Figure \ref{figResultObstacle}, we give another example of geometry where $\rcoef=0$. The waveguide contains two rather large non penetrable obstacles with Neumann boundary condition. Therefore one would expect that some energy would be backscattered. But due to the presence of the vertical branch whose height has been finely tuned, this is not the case and energy is completely transmitted ($|\mathcal{T}|=1$). In Figure \ref{figResultAngleDroit1}, we display a $\sf{L}$-shaped waveguide where $\rcoef=0$. The geometry is analogous to the one of Figure \ref{FigConclu} top right and is unbounded in the left and bottom directions. We play with the length of the diagonal branch. This framework is not exactly the one described in Section \ref{DomainOriginal2D}. However, due to the symmetry, it can be dealt with in a completely similar way. 

\begin{figure}[!ht]
\centering
\includegraphics[width=11.4cm]{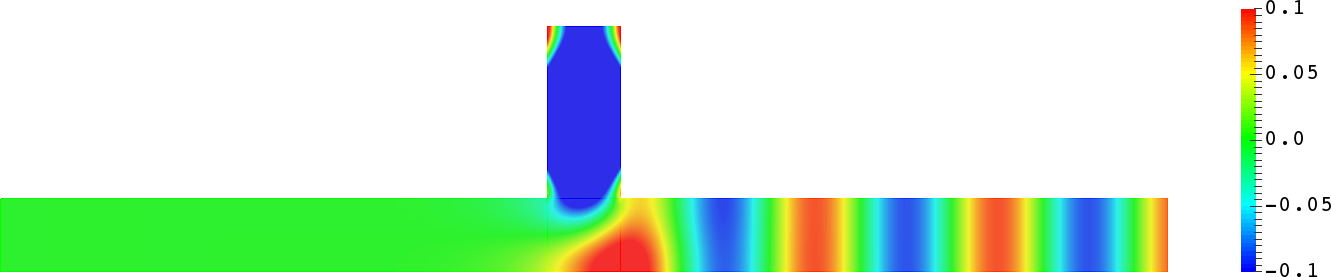}\\[14pt]\includegraphics[width=11.4cm]{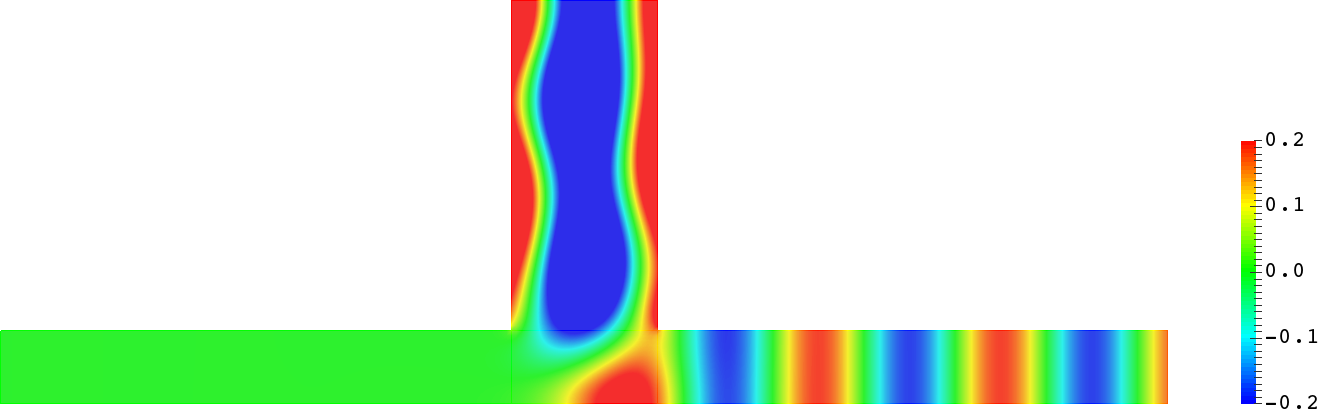} 
\caption{Real part of $v_h-w^+_h$ in two geometries where $\rcoef=0$ (non reflectivity). \label{figResult1}}
\end{figure}~

\begin{figure}[!ht]
\centering
\includegraphics[width=11.4cm]{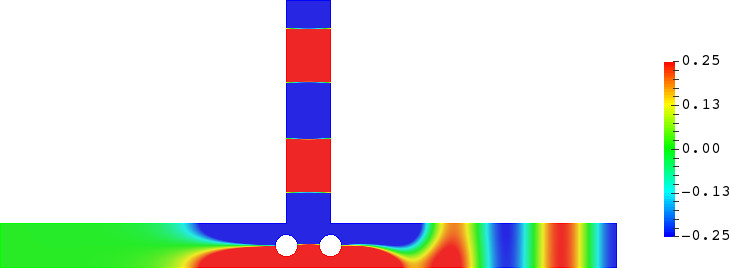}
\caption{Real part of $v_h-w^+_h$ in another geometry where $\rcoef=0$ (non reflectivity). \label{figResultObstacle}}
\end{figure}~

\begin{figure}[!ht]
\centering
\includegraphics[width=5.4cm]{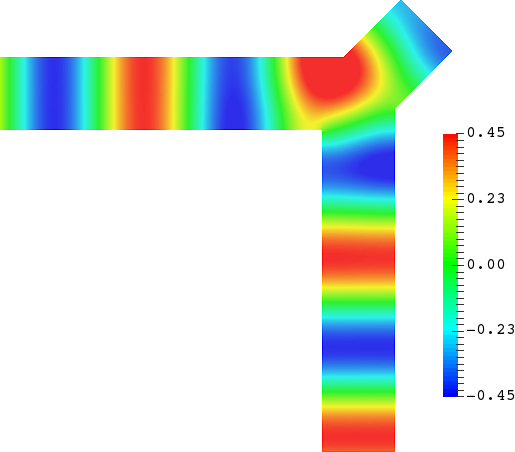}\qquad\qquad
\includegraphics[width=5.4cm]{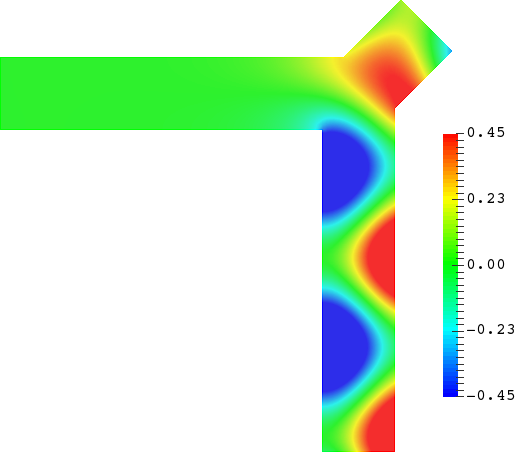}
\caption{Real parts of $v_h$ (left) and $v_h-w^+_h$ (right) in a $\sf{L}$-shaped waveguide where $\rcoef=0$ (non reflectivity).  
\label{figResultAngleDroit1}}
\end{figure}

\subsection{Perfect reflectivity}
Now in Figures \ref{figResult2} and \ref{figResultAngleDroit2}, we provide examples of waveguides where $\tcoef=0$. This time, we select the $L$ such that $-\ln|\tcoef_{h}|$ is maximum. In Figure \ref{figResult2} top, the parameters are set to $\ell=1\in(0;\pi/k)$ and $L=3.85962$ ($\tcoef_{h}\approx (-5.3+6.1i).10^{-4}$). In Figure \ref{figResult2} bottom, we take $\ell=2\in(\pi/k;2\pi/k)$ and $L=3.152073$ ($\tcoef_{h}\approx (-7.3+2.7i).10^{-4}$). As expected, the amplitude of the total field is very small in the output lead.

\begin{figure}[!ht]
\centering
\includegraphics[width=11.4cm]{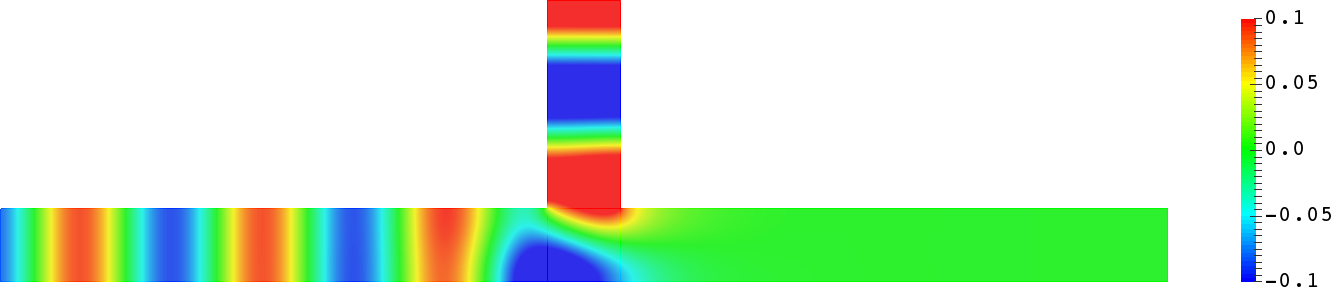}\\[14pt]\includegraphics[width=11.4cm]{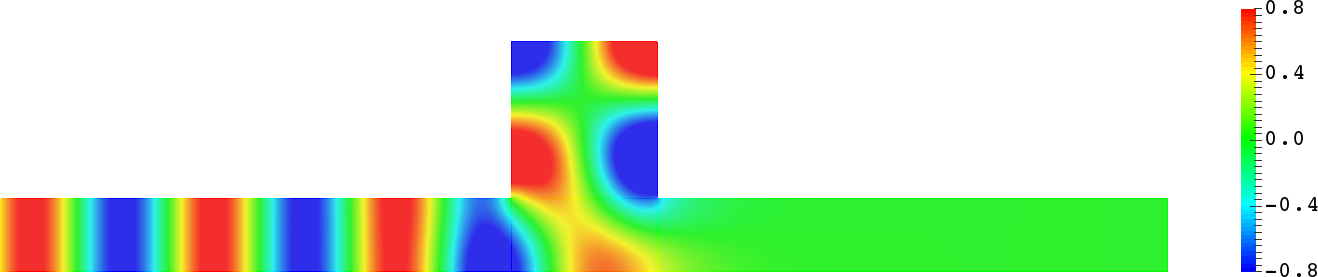} 
\caption{Real part of the total field $v_h$ in geometries where $\tcoef=0$ (perfect reflectivity).  
\label{figResult2}}
\end{figure}

\begin{figure}[!ht]
\centering
\includegraphics[width=5.4cm]{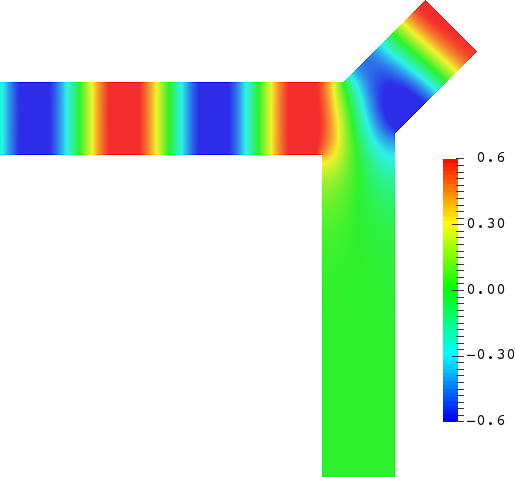}
\caption{Real part of the total field $v_h$ in a $\sf{L}$-shaped  waveguide (similar to the one of Figure \ref{figResultAngleDroit1}) where $\tcoef=0$. The geometry is unbounded in the left and bottom directions. We tune the length of the diagonal branch. 
\label{figResultAngleDroit2}}
\end{figure}

\newpage

\subsection{Perfect invisibility}\label{NumericsPerfectInvisibility}
Finally, we show in Figure \ref{figResult3} an example of waveguide where $\tcoef=1$. We start from the domain 
\[
\Om_{L}^\gamma:=\{ (x,y)\in\R\times(0;1)\ \cup\  (-\ell/2;\ell/2)\times [1;L)\ \cup\ (\pm\vartheta-1/2;\pm\vartheta+1/2)\times [1;\gamma).
\]
Numerically, first we set $\ell\in(0;\pi/k)$ and $\vartheta=1.5$ (see the notation in Section \ref{SectionCompleteInvisibility}, Figure \ref{TwoChem}). Then we approximate the solution of Problem (\ref{PbUnbounded}) (half-waveguide problem with mixed boundary conditions) for $\gamma\in(1;10)$. We select one $\gamma=\gamma_{\infty}$ such that $-\ln|R_{\infty,\,h}+1|$ is maximum. Thus we impose  $R_{\infty,\,h}\approx-1$. Eventually, we approximate the solution of the initial Problem (\ref{PbInitial}) set in $\Om^{\gamma_{\infty}}_L$ for $L\in(5;10)$ and we take $L$ such that $-\ln|\tcoef_{h}-1|$ is maximum. We try to cover a range of (relatively) high values of $L$ so that $R_{h}$ remains close to $R_{\infty,\,h}\approx-1$. Indeed, we remind the reader that the error $|\rD-R_{\infty}|$ is exponentially small as $L\to+\infty$ (see (\ref{Resultat0mode})). The parameters are set to $\ell=1\in(0;\pi/k)$, $\gamma=2.4959$ and $L=6.384936$. As expected, the amplitude of the field $v_h-w^+_h$ is very small both in the input and output leads. Numerically, we obtain a scattering coefficient $\tcoef_{h}$ such that $\tcoef_{h}-1\approx (-6.5+1.4i).10^{-6}$.

\begin{figure}[!ht]
\centering
\includegraphics[width=11.4cm]{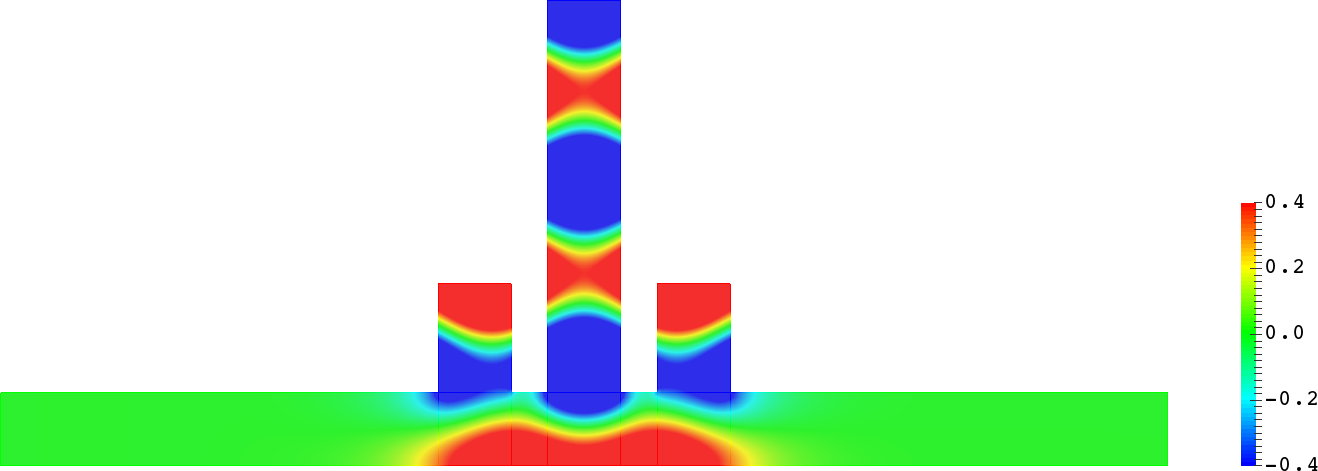}\\[10pt]\includegraphics[width=11.4cm]{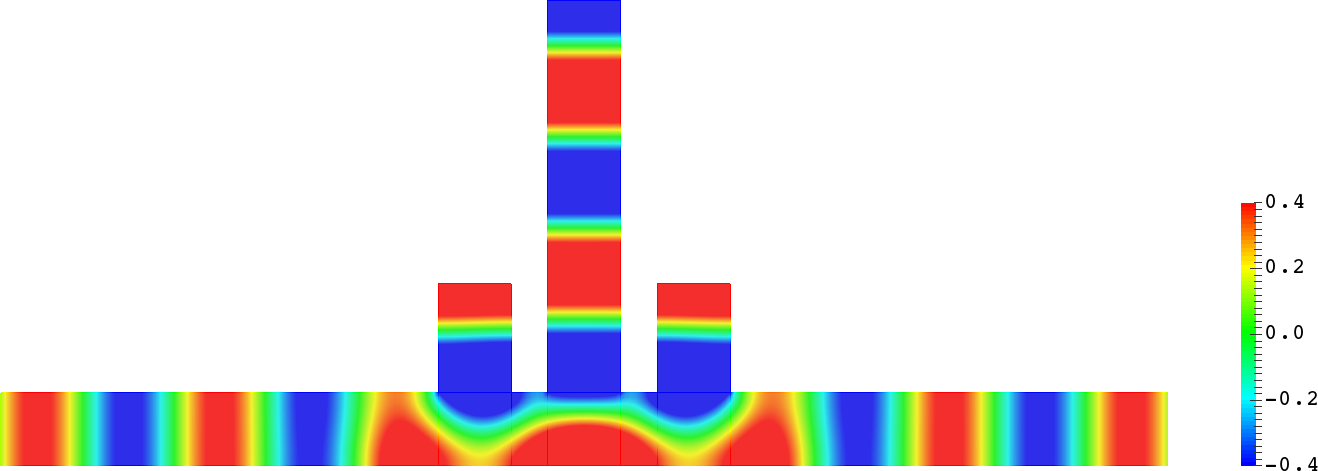} 
\caption{Example of geometry where $\tcoef=1$ (perfect invisibility). Top: real part of $v_h-w^+_h$. Bottom: real part of the total field $v_h$.
\label{figResult3}}
\end{figure}
\newpage
\section{Concluding remarks}\label{sectionConclusion}
\begin{figure}[!ht]
\centering
\raisebox{2mm}{\begin{tikzpicture}[scale=1.4]
\draw[fill=gray!30,draw=none](-2,-0.5) rectangle (2,0.5);
\draw[fill=gray!30,rotate=-60](-2,-0.2) rectangle (0,0.2);
\draw[fill=gray!30,rotate=60](0,-0.2) rectangle (2,0.2);
\draw[fill=gray!30,draw=none](-2,-0.5) rectangle (2,0.5);
\draw[fill=white] (0,0) ellipse (0.4 and 0.3);
\draw[dashed] (2.5,0.5)--(2,0.5); 
\draw[dashed] (2.5,-0.5)--(2,-0.5); 
\draw[dashed] (-2.5,0.5)--(-2,0.5); 
\draw[dashed] (-2.5,-0.5)--(-2,-0.5); 
\draw (-2,0.5)--(-0.52,0.5);
\draw (-0.06,0.5)--(0.06,0.5);
\draw (2,0.5)--(0.52,0.5);
\draw (-2,-0.5)--(2,-0.5);
\draw[dotted,>-<,rotate=60] (0,0)--(2.05,0);
\draw[dotted,>-<,rotate=120] (0,0)--(2.05,0);
\draw[dotted,>-<,rotate=120] (0,0)--(2.05,0);
\node at (0.6,1){\small $L$};
\node at (-0.6,1){\small $L$};
\end{tikzpicture}}\qquad\qquad\begin{tikzpicture}[scale=1.2]
\draw[fill=gray!30](-0.4,-0.4) rectangle (0.4,2);
\draw[fill=gray!30,draw=none,rotate=20](-2,-0.5) rectangle (0,0.5);
\draw[fill=gray!30,draw=none,rotate=-20](0,-0.5) rectangle (2,0.5);
\draw[fill=white](-0.5,-0.4) rectangle (0.5,0.1);
\draw[dashed,rotate=-20] (2.5,0.5)--(2,0.5); 
\draw[dashed,rotate=-20] (2.5,-0.5)--(2,-0.5); 
\draw[dashed,rotate=20] (-2.5,0.5)--(-2,0.5); 
\draw[dashed,rotate=20] (-2.5,-0.5)--(-2,-0.5); 
\draw[rotate=20] (-2,0.5)--(-0.24,0.5);
\draw[rotate=-20] (2,0.5)--(0.24,0.5);
\draw[rotate=20] (-2,-0.5)--(-0.17,-0.5);
\draw[rotate=-20] (2,-0.5)--(0.17,-0.5);
\draw[dotted,>-<] (0,0.32)--(0,2.05);
\node at (0.2,1.2){\small $L$};
\end{tikzpicture}\\[20pt]
\begin{tikzpicture}[scale=1.4]
\draw[fill=gray!30](-0.1,0) rectangle (0.1,2.5);
\draw[fill=gray!30] (0,0.8) circle (0.22);
\draw[fill=gray!30] (0,1.3) circle (0.22);
\draw[fill=gray!30] (0,1.8) circle (0.22);
\draw[fill=gray!30] (0,2.3) circle (0.22);
\draw[fill=gray!30,draw=none](-0.09,0) rectangle (0.09,2.5);
\draw[fill=gray!30,draw=none](-2,-0.5) rectangle (2,0.5);
\draw[dashed] (2.5,0.5)--(2,0.5); 
\draw[dashed] (2.5,-0.5)--(2,-0.5); 
\draw[dashed] (-2.5,0.5)--(-2,0.5); 
\draw[dashed] (-2.5,-0.5)--(-2,-0.5); 
\draw[dashed] (-2.5,0.5)--(-2,0.5); 
\draw (-2,-0.5)--(2,-0.5); 
\draw (-2,0.5)--(-0.1,0.5); 
\draw (2,0.5)--(0.1,0.5); 
\draw[dotted,>-<] (0,0.45)--(0,2.58);
\node at (0.4,1.5){\small $L$};
\end{tikzpicture}
\caption{Top: non perpendicular branches. Bottom: periodic vertical branch. \label{FigConclu}} 
\end{figure}
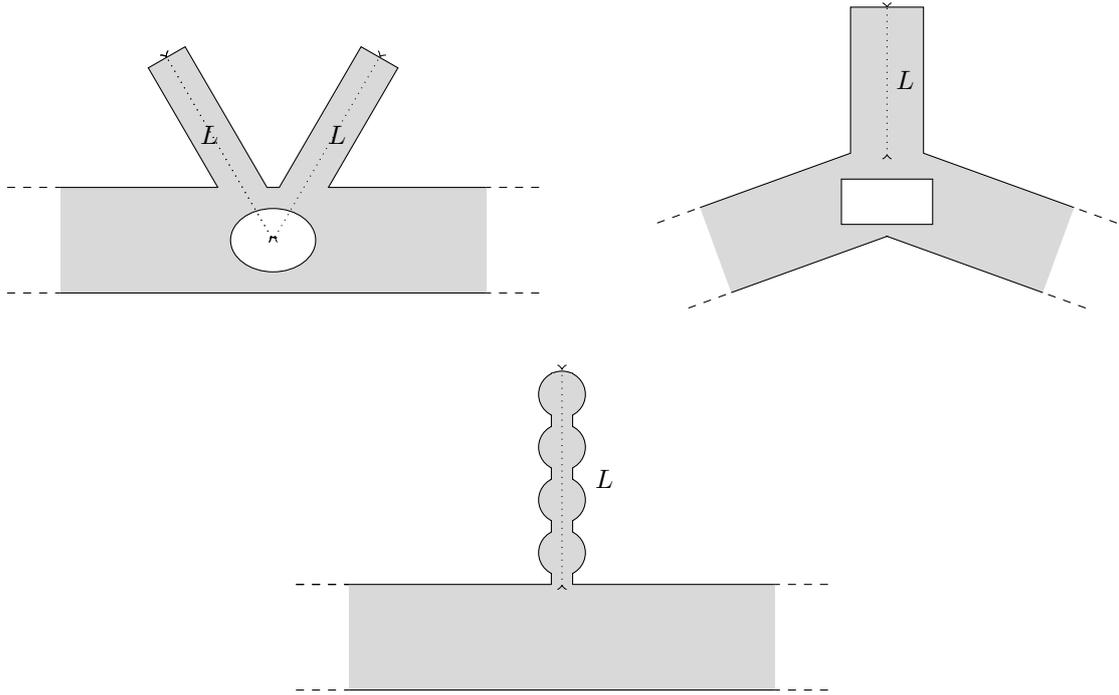

In this article, we have explained how to construct waveguides where non reflectivity ($\rcoef=0$, $|\tcoef|=1$), perfect reflectivity ($|\rcoef|=1$, $\tcoef=0$) or perfect invisibility ($\rcoef=0$, $\tcoef=1$) hold. To proceed, we have worked in geometries which are symmetric with respect to the $(Oy)$ axis and which have one or several vertical branch(es) of finite length. In our presentation, we have always assumed that the branch of finite length is perpendicular to the principal waveguide. Such an assumption is not needed and situations like the ones of Figure \ref{FigConclu} top can be considered. We could also investigate settings with truncated periodic waveguides as described in Figure \ref{FigConclu} bottom as long as waves can propagate in the vertical branch. Again we mention that we have considered only $\mrm{2D}$ problems with Neumann boundary condition but higher dimension with other boundary conditions can be dealt with using exactly the same procedure. The analysis we have presented works for the monomode regime ($0<k<\pi$). It seems complicated to extend it to configurations where several propagating modes exist in the horizontal waveguide.

\section*{Appendix}
In this Appendix, we gather the proofs of two results used in the preceding analysis.
\begin{proposition}\label{PropoAsymptotic}
Let $R$, $R_{\infty}\in\Cplx$ denote the reflection coefficients appearing respectively in (\ref{defZetaLanti}) and in (\ref{ScatteringDemiMixed}). Assume that Problem (\ref{PbUnbounded}) admits only the zero solution in $\mH^1(\om_{\infty})$ (absence of trapped modes). Then there is a constant $C>0$ independent of $L$ such that 
\begin{equation}\label{MainEstimationCoef}
|\rD - R_{\infty} | \le C\,e^{-\sqrt{(\pi/\ell)^2-k^2} L}.
\end{equation}
\end{proposition}
\begin{proof}
Set $\xi:=2\max(\ell,d)$ and define the segment $\Upsilon:=\{-\xi\}\times(0;1)$. Note that the domains $\om_L$, $\om_\infty$ coincide with the strip $\R\times(0;1)$ for $x\le-\xi$. From the expansions (\ref{defZetaLanti}) and (\ref{ScatteringDemiMixed}) for $U$ and $U_{\infty}$, using decomposition in Fourier series, one finds 
\begin{equation}\label{CalculusDiffCoef}
\rD - R_{\infty}=2k\int_{\Upsilon}(U-U_{\infty})\,w^+\,dy.
\end{equation}
Define the domain $D_L:=\{(x,y)\in\om_L\,|\,x\ge-\xi\}$. From the continuity of the trace operator from $\mH^1(D_L)$ to $\mL^2(\Upsilon)$ (with a constant of continuity independent of $L$), together with (\ref{CalculusDiffCoef}), we see that to establish (\ref{MainEstimationCoef}), it is sufficient to show the estimate
\begin{equation}\label{VolumicEstimate}
\|U - U_{\infty} \|_{\mH^1(D_L)} \le C\,e^{-\sqrt{(\pi/\ell)^2-k^2} L}.
\end{equation}
Here and in what follows $C>0$ is a constant which can change from one line to another but which is independent of $L$. Now we focus our attention on the proof of (\ref{VolumicEstimate}).\\
\newline
First we introduce some notation to reformulate Problem (\ref{PbChampTotalAntiSym}) in the bounded domain $D_L$. Introduce the Dirichlet-to-Neumann map 
\[
\begin{array}{lccl}
\Lambda:&\mH^{1/2}(\Upsilon)&\to&\mH^{-1/2}(\Upsilon)\\
 &\psi & \mapsto & \Lambda(\psi)=-\partial_x W|_{\Upsilon},
\end{array}
\]
where $W$ is the unique function satisfying 
\[
\begin{array}{|rcll}
\Delta W + k^2 W & = & 0 & \mbox{ in }(-\infty;-\xi)\times(0;1)\\[3pt]
 \partial_yW  & = & 0  & \mbox{ on }\partial((-\infty;-\xi)\times(0;1))\setminus\Upsilon \\[3pt]
W  & = & \psi  & \mbox{ on }\Upsilon
\end{array}
\]
and admitting the expansion $W=c\,w^-+\tilde{W}$. Here $c\in\Cplx$ and  $\tilde{W}\in\mH^1((-\infty;-\xi)\times(0;1))$. One can check that $\Lambda$ is well-defined. Moreover, it is known that $\Lambda$ is a linear and continuous map from $\mH^{1/2}(\Upsilon)$ to $\mH^{-1/2}(\Upsilon)$. Define the space $\mX:=\{V\in\mH^1(D_L)\,|\,V=0\mbox{ on }\Sigma_L\}$. If $U$ is a solution of (\ref{PbChampTotalAntiSym}) admitting expansion (\ref{defZetaLanti}), then $U$ solves the problem 
\begin{equation}\label{PbVaria}
\begin{array}{|lcl}
\mbox{Find $U\in\mX$ such that}\\
a(U,V)=-2ik\dsp\int_{\Upsilon}w^{+}\overline{V}\,dy,\qquad \forall V\in\mX,
\end{array}
\end{equation}
with $a(U,V)=\int_{D_L}\nabla U\cdot\overline{\nabla V}-k^2U\,\overline{V}\,dxdy-\langle \Lambda(U),\overline{V}\rangle_{\Upsilon}$. Here $\langle \cdot,\cdot\rangle_{\Upsilon}$ stands for the (bilinear) duality pairing between $\mH^{-1/2}(\Upsilon)$ and $\mH^{1/2}(\Upsilon)$. Conversely, if $U$ satisfies (\ref{PbVaria}), one can extend $U$ as a solution of (\ref{PbChampTotalAntiSym}) admitting expansion (\ref{defZetaLanti}). With the Riesz representation theorem, introduce the bounded operator $\mA(L):\mX\to\mX$ such that 
\[
(\mA(L)U,V)_{\mH^1(D_L)}=a(U,V),
\]
where $(\cdot,\cdot)_{\mH^1(D_L)}$ denotes the usual inner product of $\mH^1(D_L)$. Under the assumption that the only solution of Problem (\ref{PbUnbounded}) in $\mH^1(\om_{\infty})$ is zero (absence of trapped modes), one can show that $\mA(L)$ is invertible for $L$ large enough. Moreover there is a constant $C>0$ independent of $L$ large enough such that 
\begin{equation}\label{StabilityEstimate}
\|\mA(L)^{-1}\| \le C.
\end{equation}
Here $\|\cdot\|$ stands for the usual norm on the set of linear operators of $\mX$. For the proof of this non trivial stability estimate, we refer the reader to \cite[Chap. 5, \S5.6, Thm. 5.6.3]{MaNP00}.\\ 
\newline
We come back to the proof of (\ref{VolumicEstimate}). For $L$ large enough, using (\ref{StabilityEstimate}), we can write 
\begin{equation}\label{sequenceEstimates0}
 \|U - U_{\infty} \|_{\mH^1(D_L)} = \| \mA(L)^{-1}(\mA(L)(U - U_{\infty})) \|_{\mH^1(D_L)} \le  C\,\|\mA(L)(U - U_{\infty}) \|_{\mH^1(D_L)}.
\end{equation}
Observing that 
\[
(\mA(L)U_{\infty},V)_{\mH^1(D_L)}=a(U_{\infty},V)=-2ik\dsp\int_{\Upsilon}w^{+}\overline{V}\,dy+\int_{\Gamma_L}\partial_yU_{\infty}\,\overline{V}\,dx
\]
(we remind the reader that $\Gamma_L=(-\ell/2;0)\times\{L\}$), we deduce that 
\begin{equation}\label{sequenceEstimates1}
\|U - U_{\infty} \|_{\mH^1(D_L)} \le  C\,\sup_{V\in\mX\setminus\{0\}}\Big|\int_{\Gamma_L}\partial_yU_{\infty}\,\overline{V}\,dx\Big|/\|V\|_{\mH^1(D_L)}.
\end{equation}
Now we assess the right hand side of (\ref{sequenceEstimates1}). To proceed, again we need to introduce some notation. For $n\in\N:=\{0,1,\dots\}$, define the  function $\varphi_n$ such that $\varphi_n(x)=2\,\ell^{-1/2}\sin(\pi(1+2n)x/\ell)$. The family $(\varphi_n)$ is an orthonormal basis of $\mL^2(\Gamma_L)$. Set $\lambda_n=\pi(1+2n)/\ell$. Using the equation satisfied by $U_{\infty}$, we obtain the decomposition,  for $y\ge d$, 
\begin{equation}\label{ExplicitExpansion}
U_{\infty}(x,y)=\sum_{n=0}^{\infty}\alpha_n\,e^{-\beta_n y}\varphi_n(x)\  \mbox{ with }\ \beta_n:=\sqrt{\lambda^2_n-k^2}\ \mbox{ and }\   \alpha_n=e^{\beta_n d}\int_{\Gamma_L}U_{\infty}\,\varphi_n \,dx.
\end{equation}
On the other hand, classical results guarantee that for $V\in\mX$, the norm $\|V\|_{\mH^{1/2}(\Gamma_L)}$ is equivalent to
\begin{equation}\label{NormEquiv}
\left(\sum_{n=0}^{+\infty}(1+|\lambda_n|^{1/2})^2\,|a_n|^2\right)^{1/2}
\end{equation}
with $a_n:=\dsp\int_{\Gamma_L} V\,\varphi_n\,dx$. This allows us to write
\[
\Big|\int_{\Gamma_L}\partial_yU_{\infty}\,\overline{V}\,dx\Big| \le C\,\|U_{\infty}\|_{\mH^{1/2}(\Gamma_L)}\,\|V\|_{\mH^{1/2}(\Gamma_L)}\le C\,\|U_{\infty}\|_{\mH^{1/2}(\Gamma_L)}\,\|V\|_{\mH^1(D_L)}.
\]
Plugging the latter estimate in (\ref{sequenceEstimates1}), we deduce 
\begin{equation}\label{sequenceEstimates}
\|U - U_{\infty} \|_{\mH^1(D_L)} \le  C\,\|U_{\infty}\|_{\mH^{1/2}(\Gamma_L)}.
\end{equation}
Working on the expansion (\ref{ExplicitExpansion}) with formula (\ref{NormEquiv}), we can show that
\begin{equation}\label{EstimateTraces}
\|U_{\infty}\|_{\mH^{1/2}(\Gamma_L)}\le e^{\beta_0(d-L)}\|U_{\infty}\|_{\mH^{1/2}(\Gamma_d)}. 
\end{equation}
Since $\beta_0=\sqrt{(\pi/\ell)^2-k^2}$, plugging (\ref{EstimateTraces}) in (\ref{sequenceEstimates}), finally we find 
\[
\|U - U_{\infty} \|_{\mH^1(D_L)} \le  C\,e^{\sqrt{(\pi/\ell)^2-k^2}(d-L)}\|U_{\infty}\|_{\mH^{1/2}(\Gamma_d)},
\]
which is nothing but estimate (\ref{VolumicEstimate}).
\end{proof}
\begin{proposition}\label{ProofUnitary}
The scattering matrix $\mathfrak{s}_{\infty}$ defined in (\ref{UnboundedScatteringMatrix}) is unitary and symmetric. 
\end{proposition}
\begin{proof}
Define the symplectic (sesquilinear and anti-hermitian ($q(\varphi,\psi)=-\overline{q(\psi,\varphi)}$)) form $q(\cdot,\cdot)$ such that for all $\varphi,\psi\in\mH^1_{\loc}(\om_\infty)$
\[
q(\varphi,\psi)=\int_{\Sigma} \cfrac{\partial \varphi}{\partial n}\,\overline{\psi}-\varphi\cfrac{\partial \overline{\psi}}{\partial n}\,d\sigma.
\]
Here $\Sigma:=\{-\xi\}\times(0;1)\cup (-\ell;0)\times\{\xi\}$, $\partial_n=-\partial_x$ on $\{-\xi\}\times(0;1)$, $\partial_n=\partial_y$ on $(-\ell;0)\times\{\xi\}$ and $\xi$ is a given parameter chosen large enough. Moreover, $\mH^1_{\loc}(\om_\infty)$ refers to the Sobolev space of functions $\varphi$ such that $\varphi|_{\mathcal{O}}\in\mH^1(\mathcal{O})$ for all bounded domains $\mathcal{O}\subset\om_\infty$. Integrating by parts and using that the functions $u^-_{\infty}$, $u^\circ_{\infty}$ defined in (\ref{decompoUnbounded}) satisfy the Helmholtz equation, we obtain $q(u_{\infty}^i,u_{\infty}^j)=0$ for $i,\,j\in\{-,\circ\}$. On the other hand, decomposing $u_{\infty}^-$, $u_{\infty}^\circ$ in Fourier series on $\Sigma$, we find 
\[
\begin{array}{c}
q(u_{\infty}^-,u_{\infty}^-) = (-1+|r_{\infty}|^2+|t_{\infty}|^2)\,i,\quad q(u_{\infty}^\circ,u_{\infty}^\circ) = (-1+|r^{\circ}_{\infty}|^2+|t^{\circ}_{\infty}|^2)\,i\\[5pt]
q(u_{\infty}^-,u_{\infty}^\circ)=-\overline{q(u_{\infty}^\circ,u_{\infty}^-)}=r_{\infty}\overline{t^{\circ}_{\infty}}+t_{\infty}\overline{r_{\infty}^{\circ}}.
\end{array}
\]
These relations allow us to prove that $\mathfrak{s}_{\infty}\,\overline{\mathfrak{s}_{\infty}}^{\top}=\mrm{Id}_{2\times 2}$, that is to conclude that $\mathfrak{s}_{\infty}$ is unitary. On the other hand, one finds $q(u_{\infty}^-,\overline{u_{\infty}^\circ})=0=-t_{\infty}^{\circ}+t_{\infty}$. We deduce that $\mathfrak{s}_{\infty}$ is symmetric.
\end{proof}

\section*{Acknowledgments}
The research of S.A. N. was supported by the grant No. 18-01-00325 of the Russian Foundation on Basic Research. V. P. acknowledges the financial support of the Agence Nationale de la Recherche through the Grant No. DYNAMONDE ANR-12-BS09-0027-01.

\bibliography{Bibli}
\bibliographystyle{plain}

\end{document}